\documentclass[a4paper,11pt]{article}

\usepackage[left=2.6cm, right=2.6cm, top=2.6cm, bottom=2.6cm]{geometry}

\usepackage[latin1]{inputenc}
\usepackage[T1]{fontenc}
\usepackage[english,ngerman]{babel}

\usepackage{amsthm}
\usepackage{amsmath}
\usepackage{amssymb}

\usepackage{paralist}
\usepackage[arrow, matrix, curve]{xy}

\usepackage{booktabs, array, dcolumn}
\usepackage{tabularx}

\usepackage{url}

\usepackage{float}
\usepackage{graphicx}
\usepackage{caption}
\usepackage{subcaption}
\usepackage{color}

\usepackage{dsfont}

\usepackage{verbatim}

\newtheoremstyle{j}%
{3pt}% 
{3pt}% 
{}%
{\parindent}% 
{\bfseries}% 
{.}%
{.5em}% 
{}% 

\theoremstyle{plain}

\newtheorem*{rem*}{Remark}
\newtheorem*{concl*}{Conclusion}
\newtheorem*{theorem*}{Theorem}
\newtheorem*{cor*}{Corollary}
\newtheorem*{algo*}{Algorithm}

\newtheorem{theorem}{Theorem}[section]
\newtheorem{lemma}[theorem]{Lemma}

\newtheorem{cor}[theorem]{Corollary}
\newtheorem{prop}[theorem]{Proposition}
\newtheorem{deff}[theorem]{Definition}
\newtheorem{rem}[theorem]{Remark}

\theoremstyle{definition}

\newtheorem*{example*}{Example}

\newtheorem*{prob*}{Problem}

\newcommand{\Hil}{\mathcal{H}}

\newcommand{\Fcalt}{\widetilde{\mathcal{F}}}
\newcommand{\Scal}{\mathcal{S}}

\newcommand{\Rcal}{\mathcal{R}}

\newcommand{\Z}{\mathbb{Z}}
\newcommand{\N}{\mathbb{N}}

\newcommand{\R}{\mathbb{R}}

\newcommand{\C}{\mathbb{C}}

\newcommand{\fhat}{\widehat{f}}

\newcommand{\psihat}{\widehat{\psi}}
\newcommand{\psitilde}{\widetilde{\psi}}

\newcommand{\phihat}{\widehat{\phi}}

\newcommand{\eps}{\varepsilon}

\newcommand{\bal}{\begin{align}}
\newcommand{\eal}{\end{align}}

\newcommand{\bM}{\begin{pmatrix}}
\newcommand{\eM}{\end{pmatrix}}

\DeclareMathOperator{\spann}{span}

\DeclareMathOperator{\suppp}{supp \,}

\numberwithin{equation}{section}

\newcommand{\operp}{\mathop{\bigcirc\kern-12.75pt\perp}\nolimits}

\numberwithin{equation}{section}

\allowdisplaybreaks

\newcommand*{\ti}{\textit}

\begin{document}
\selectlanguage{english}

\title{Generalized sampling reconstruction from Fourier measurements using compactly supported shearlets}
\author{
  Jackie Ma\footnote{ma@math.tu-berlin.de}\\
  Technische Universit\"{a}t Berlin\\ Department of Mathematics \\
  Stra\ss{}e des 17. Juni 136,   10623 Berlin
}

\maketitle

\begin{abstract}
In this paper we study the general reconstruction of a compactly supported function from its Fourier coefficients using compactly supported shearlet systems. We assume that only finitely many Fourier samples of the function are accessible and based on this finite collection of measurements an approximation is sought in a finite dimensional shearlet reconstruction space. We analyse this sampling and reconstruction process by a recently introduced method called \emph{generalized sampling}. In particular by studying the \emph{stable sampling rate} of generalized sampling we then show stable recovery of the signal is possible using an almost linear rate. Furthermore, we compare the result to the previously obtained rates for wavelets.
\end{abstract}

%%%%%%%%%%%%%%%%%%%
%%%%%% SECTION 1%%%%%%%
%%%%%%%%%%%%%%%%%%%

\section{Introduction}

A general task in sampling theory is the reconstruction of an object from finitely many measurements. Such a problem also appears in more applied fields such as signal processing and medical imaging, the latter being a main motivation of this paper with view to applications. A Hilbert space where such problems are often modeled in is the Hilbert space $L^2(\R^2)$ - the space of square integrable functions -, indeed, the reconstruction task in magnetic resonance imaging (MRI) is mostly modeled in  $L^2(\R^2)$. In this paper we assume the object of interest can be modeled by a compactly supported function $f \in L^2(\R^2)$ and the acquired samples are simply the Fourier coefficients of this function. Of course, a reconstruction of such a function $f$ can then be obtained by forming a truncated Fourier series using the sampled Fourier coefficients. However, physical constraints make it  impossible to sample all Fourier coefficients which causes an error in the reconstruction and possibly leads to non satisfactory reconstructions. For example the Gibbs phenomenon is a typical artifact that is strongly visible in the approximation if not enough Fourier coefficients are used in a Fourier series. Sampling Fourier coefficients and forming a reconstruction by a Fourier series means that sampling and reconstructing the signal is both done by using the same system. While the way we acquire the data is indeed usually fixed by the acquisition device, for instance in MRI, the reconstruction system is not. Very often additional information about the object is known, e.g. the shape or the regularity. Therefore it is plausible to use this extra information in the reconstruction scheme by choosing a suitable reconstruction system that favors a representation of this object.

\subsection{Sampling and reconstruction model}

As already mentioned at the beginning of the introduction, many reconstruction problems can be modeled in a Hilbert space $\Hil$ with inner product $\langle \cdot, \cdot \rangle$. The samples are then assumed to be acquired with respect to a fixed sampling system $\{s_1, s_2, \ldots \} \subset \Hil$ and the reconstruction is computed using another system $\{r_1,r_2, \ldots \} \subset \Hil$ which can be arbitrarily chosen. The concrete reconstruction problem can then be formulated as follows: Given finitely many linear measurements $\langle f, s_1 \rangle, \ldots, \langle f, s_M \rangle, M \in \N$ of a function $f \in \Hil$, find a reconstruction $f_N \in \spann \{r_1, \ldots, r_N\}, N \in \N$ that is numerically stable and whose error converges to zero as $N$ tends to infinity. 

Reconstruction problems of this type have already been studied, e.g. by Unser and Aldroubi in \cite{UnsAld} and later also by Eldar  in \cite{Eldar1, Eldar2} for the case $N=M$ where the used method is called \emph{consistent reconstruction}. We investigate this reconstruction problem by using \emph{generalized sampling} (GS) which was introduced by Adcock et al. in a series of papers \cite{AH1, AH3, AHP2}. In fact, GS can be seen as an extension of consistent reconstruction and a comparison of both methods is presented in \cite{AHP2}. The key difference of GS to consistent reconstruction is to allow $M$ and $N$ to vary independently of each other. This additional flexibility can be used to overcome some of the issues that are addressed in \cite{AHP2}. GS further provides a quantity called the \emph{stable sampling rate}, which represents the number of samples $M \in \N$ that are needed in order to obtain an approximation that is stable and convergent using the first $N \in \N$ reconstruction elements of the a priori chosen and ordered reconstruction system. We will introduce generalized sampling and in particular the stable sampling rate in sufficient detail in Section \ref{section:GeneralizedSampling}.

\subsection{Sparsifying systems}
 
The selection of a different system for reconstruction than the given sampling system can only be done faithfully if the structure of the underlying object is known. Indeed, it is well known that most natural images, in particular medical images, have a \emph{sparse representation} with respect to a \emph{wavelet basis} \cite{Dau}, meaning either many wavelet coefficients are zero or have a fast decay with respect to increasing scales. This sparse representation by wavelets is one of the reasons why they became very important in imaging problems, see e.g. \cite{JPEG}. 

However, over the years many other sparsifying systems similar to wavelets have arisen such as \emph{curvelets} \cite{CanDon} and \emph{shearlets} \cite{KutLabBook}. More importantly, these systems are known to outperform  wavelets in higher dimensions in the sense that both, curvelets and shearlets provide a so-called \emph{almost optimal sparse approximation rate} of \emph{cartoon-like} images which are a standard model for natural and medical images, see \cite{DonSparse, CanDon, KutLim}, while wavelets provable do not reach this rate.  For this reason, it is of utmost importance to study not only wavelets but in addition shearlets as a reconstruction system to approximate a compactly supported function based on its Fourier coefficients in the context of generalized sampling. Indeed, wavelets has been studied in \cite{AHP1,AHKM} see also Subsection \ref{sec:contribution}. 

Note that there exists a general notion of \emph{parabolic molecules} introduced by Grohs et al. \cite{ParMol} for which shearlets and curvelets are special instances. A natural question is then, whether our results can be extended to such systems and we will comment on this in Section \ref{sec:futurework}.

\subsection{Contribution and related work}\label{sec:contribution}

To enhance the quality of the Fourier inversion of the sampled data, great effort has been done in applications to improve the reconstruction process by changing the reconstruction basis or adapting the sampling process. For example in \cite{HealyWeaver} the authors introduced \emph{wavelet-encoding} which allows the user to directly sample wavelet coefficients instead of Fourier coefficients by adapting the acquisition device. However, as for now this method also has several shortcomings such as low signal-to-noise ratio \cite{Panych}.  On the other hand, the proposed method GS can be understood as a post processing method that does not require to change the acquisition device since the sparsifying transform will be incorporated in the reconstruction process and not in the sampling process. 

The use of GS to study the reconstruction problem from Fourier measurements has already been investigated for wavelets in 1D \cite{AHP1} and 2D \cite{AHKM}. The authors of the respective works showed, that for dyadic multiresolution analysis (MRA) wavelets the stable sampling rate is linear, meaning up to a constant one has to sample as many Fourier coefficients as one wants to reconstruct wavelet coefficients. Moreover, it was shown in \cite{AHP1} that a defiance of this rate leads to unstable approximations. 

The main result of our paper shows that the stable sampling rate is almost linear for sampling Fourier coefficients and reconstructing shearlet coefficients. The main challenge compared to the previous results of compactly supported MRA wavelets is that shearlets neither obey a classical scaling equation which enables an expansion of translations of a single function at highest scale nor form an orthonormal bases. However, both properties are heavily used in the proofs for the wavelet case. Therefore the proof requires a different strategy and is build upon good localization properties in frequency of the reconstruction system. Moreover, to the best of the authors knowledge there are so far no results for a stable sampling rate for redundant reconstruction systems that merely form a frame. In fact, the redundancy or instability measured  by the lower frame bound of the reconstruction system now becomes apparent and effects the stable sampling rate directly. Such a penalty does not occur when using orthonormal bases or Riesz bases. Therefore the rate that we obtain for shearlets is worse than the rate for wavelets.

Since the stable sampling rate is worse for shearlets than for wavelets we justify the theoretical gain of using shearlets in combination of GS by showing that the error of the GS shearlet reconstruction converges asymptotically faster to zero than the error of the GS wavelet reconstruction under some assumptions if both reconstructions are computed from the same number of measurements.

As for practical applications such as MRI, wavelets are often used as a reconstructions system and it has been shown that these often lead to superior results in terms of lower sampling rates and reduced artifacts  \cite{LusDonPau}. We will provide some numerical tests with MR data and deterministic subsampling patterns in the numerics section using shearlets and compare these to reconstruction methods based on wavelets and Fourier inversion to show the strength of using shearlets as a reconstruction system for the recovery problem from Fourier measurements.

\subsection{Notation}

In GS the sampling system and the reconstruction system are allowed to be almost completely arbitrary, in fact, assuming the frame property is enough. Recall that a countable sequence $(f_i)_{i \in I}$ is called a \emph{frame} (\cite{Chr}) for a Hilbert space $\Hil$ with inner product $\langle \cdot,\cdot \rangle$ if there exist $0<A\leq B< \infty$ such that
\begin{align*}
	A \| f \|^2 \leq \sum_{i \in I} | \langle f, f_i \rangle |^2 \leq B \|f \|^2, \quad \forall f \in \Hil.
\end{align*}
For most parts we will assume $\Hil$ to be the space of square integrable functions $L^2(\R^2)$ equipped  with the standard inner product $\langle f, g \rangle := \int f \overline{g} \, d \mu$ and induced norm $\| f \| := \sqrt{\langle f, f \rangle}$ for $f,g \in L^2(\R^2)$.

The Fourier transform of a function $f \in L^1(\R^2)$ is denoted by $\fhat$ and we use the definition
\begin{align*}
	\fhat(\xi) = \int_{\R^2} f(x) e^{-2 \pi i \langle x, \xi \rangle} \, d x, \quad \xi \in \R^2,
\end{align*}
and with a standard approximation argument we extend this definition to functions in $L^2(\R^2)$.

We will also often write "$\lesssim$" to indicate an inequality to hold up to some constant that is irrelevant or independent of all important parameters. Similarly "$\asymp$" should denote equality up to some constant.

\subsection{Outline}

We continue with a short review of generalized sampling in Section \ref{section:GeneralizedSampling}. Then we introduce shearlets and recall the concept of sparse approximation of cartoon-like functions in sufficient detail in Section \ref{section:Shearlets}. Section \ref{section:SampRecSpace} contains the precise definitions of the sampling and reconstruction spaces that are used for the main theorem that is formulated in Subsection \ref{section:MainResult}.  All proofs are then presented in Section \ref{section:Proofs}. Finally, in Section \ref{section:Numerics} we aim to numerically confirm the methodology and effectiveness of using  shearlets as a reconstruction system.

\section{Generalized sampling} \label{section:GeneralizedSampling}

Generalized sampling was introduced by Adcock et al. in  \cite{AH1, AH3, AHP2} as a reconstruction method for almost arbitrary sampling and reconstruction systems. In fact, similar to \emph{consistent reconstruction} \cite{UnsAld, Eldar1, Eldar2} these systems are assumed to form a frame. As already mentioned in the introduction, the key difference of generalized sampling and consistent reconstruction is to allow the number of measurements to vary independently of the number of reconstruction elements. This flexibility enables generalized sampling to overcome some of the barriers of consistent reconstruction \cite{AH1}.

\subsection{GS reconstruction method}

Given a \emph{sampling system} $\{s_1, s_2, \ldots \} \subset \Hil$ we define the sampling space $\Scal \subset \Hil$ as the closure of its span, i.e.
\begin{align*}
	\Scal = \overline{\spann} \{ s_k \, : \, k \in \N \}.
\end{align*}
The finite dimensional version of this sampling space is denoted by
\begin{align*}
	\Scal_M = \spann \{ s_1, \ldots, s_M \}, \quad M \in \N.
\end{align*}
The sampling vectors $\{ s_k \, : \,  k\in \N\}$ are used to model the measurements, in fact, we assume that the samples are given as linear measurements of the form 
\begin{align}
	m_f(k) = \langle f, s_k \rangle, \quad k \in \N. \label{eq:Measurements}
\end{align}
Analogously for a \emph{reconstruction system} $\{r_1,r_2, \ldots \} \subset \Hil$ we define the reconstruction space $\Rcal \subset \Hil$ as
\begin{align*}
	\Rcal = \overline{\spann} \{ r_k \, : \, k \in \N \}
\end{align*}
and, likewise, its finite dimensional version is defined as
\begin{align*}
	\Rcal_N = \spann \{ r_1, \ldots, r_N \}, \quad N \in \N. 
\end{align*}
We will assume the sampling system to form an orthonormal basis and the reconstruction system to be a frame, since this will exactly be our setup, see Section \ref{section:SampRecSpace}. A fully general presentation of generalized sampling can be found in \cite{AH1, AHP2}. Further it is natural to assume a certain subspace condition, indeed, it is required that
\begin{align}
	\Rcal \cap \Scal^\perp = \{ 0 \} \quad \text{and} \quad \Rcal + \Scal \text{ is closed.} \label{eq:SubspaceCondition}
\end{align}
This guarantees a well posedness of the finite dimensional reconstruction problem, cf. Theorem \ref{thm:Existence}. Indeed, the reconstruction problem is now formulated as follows: given a finite number of measurements $\langle f, s_1 \rangle, \ldots, \langle f, s_M \rangle$ of some unknown $f \in \Hil$ we wish to determine a reconstruction $f_N \in \Rcal_N$ such that $\| f - f_N \|$ is small and $f_N \to f$ as $N \to \infty$ fast. The following theorem leads to the generalized sampling reconstruction and proves its existence.

\begin{theorem}[\cite{AHP2}]\label{thm:Existence}
Let $\Scal_M$ and $\Rcal_N$ be as above and $P_{\Scal_M}$ be the following finite rank operator
\begin{align*}
	P_{\Scal_M} : \Hil &\to \Scal_M \\
	f & \mapsto \sum_{k=1}^M \langle f, s_k \rangle  s_k.
\end{align*}
If \eqref{eq:SubspaceCondition} holds, then there exists an $M \in \N$ such that the system of equations
\begin{align}
	\langle P_{\Scal_M} f_{N,M}, r_j \rangle = \langle P_{\Scal_M} f , r_j\rangle, \quad j = 1, \ldots, N \label{eq:thmEq}
\end{align}
has a unique solution $f_{N,M } \in \Rcal_N$. Moreover, the smallest $M \in \N$ such that the system is uniquely solvable is the least number $M \in \N$ so that
\begin{align}
	c_{N,M} := \inf_{\substack{f \in \Rcal_N,\\ \| f \| = 1}} \| P_{\Scal_M} f \| > 0. \label{eq:thmExistence}
\end{align}
Furthermore,
\begin{align}
	\| f - P_{\Rcal_N} f \| \leq \| f - f_{N,M} \| \leq \frac{1}{c_{N,M}} \| f - P_{\Rcal_N} f \|, \label{eq:thmError}
\end{align}
where $P_{\Rcal_N}: \Hil \to \Rcal_N$ denotes the orthogonal projection onto $\Rcal_N$.
\end{theorem}

\begin{deff}[\cite{AHP2}]
The solution $f_{N,M}$ in Theorem \ref{thm:Existence} is called \emph{generalized sampling reconstruction}.
\end{deff}

\begin{deff}[\cite{UnsAld, Tang}]
The quantity $c_{N,M}$ in \eqref{eq:thmExistence} is called the \emph{infimum cosine angle} between the subspaces $\Rcal_N$ and $\Scal_M$.
\end{deff}

\subsection{Stable sampling rate}

In light of Theorem \ref{thm:Existence} it is crucial to have control over $c_{N,M}$. This motivates the definition of the stable sampling rate.

\begin{deff}\label{def:SSR}
For any fixed $N \in \N$ and $\theta >1$ the \emph{stable sampling rate} $\Theta(N, \theta)$ is defined as
\begin{align*}
	\Theta(N,\theta) = \min \left \{M \in \N \, : \, c_{N,M} > \frac{1}{\theta} \right\}.
\end{align*}
\end{deff}

The stable sampling rate determines the number of measurements $M$ that are needed in order to find an approximation $f_{N,M} \in \Rcal_N$ or equivalently, find $N$ coefficients such that the angle $c_{N,M}$ is controlled by the threshold $\theta$. We wish to emphasize that the control of $c_{N,M}$ allows the control of the error that is made by the approximation via GS by  \eqref{eq:thmError}.

%\subsection{Computational aspects of GS}
%
%The generalized sampling reconstruction can be computed from the least squares solution of 
%\begin{align}
%U^{[M,N]} x^{[N]} = m(f)^{[M]},  \label{eq:leastsq}
%\end{align}
%where
%$$    
%U^{[M,N]}= \begin{pmatrix}
%        u_{11} & \ldots & u_{1N} \\
%        \vdots & \ddots & \vdots   \\
%        u_{M1} & \ldots & u_{MN}
%    \end{pmatrix}, \; u_{ij} = \langle r_j, s_i \rangle, \quad m_f^{[M]} = (m_f(1), \ldots, m_f(M)), 
%$$
%The solution vector $x^{[N]}$ consists of the sought reconstruction coefficients.
%
%
%In \cite{AHP2} the authors showed the computability of the stable sampling rate if the reconstruction systems forms a basis for its span. In particular, the following result holds.
%\begin{lemma}[\cite{AHP2}]\label{lemma:SSR}
%Let $\{r_j\}_{j=1}^N$ be a basis for $\Rcal_N$. Then $c_{N,M}$ is the square root of the minimal generlized eigenvalue of the matrix pencil $\{ (U^{[N,M]})*, A^{[N,N] \}$ where $A^{[N,N]$ is the gramian for 
%\end{lemma}

\section{Shearlets}\label{section:Shearlets}

In this section we introduce the function systems that are used for the reconstruction space, these are \emph{compactly supported shearlets}. 

Shearlet systems  were first introduced by K. Guo, G. Kutyniok, D. Labate, W.-Q Lim and G. Weiss in \cite{GuoKutLab2006, LLKW2007} and the following notations have became standard in this topic. The  \emph{parabolic scaling matrices} with respect to scale $j \in \N \cup \{ 0 \}$ are denoted by
\begin{align*}
	A_{2^j} := \begin{pmatrix} 2^j & 0 \\ 0 & 2^{j/2 } \end{pmatrix}, \qquad \widetilde{A}_{2^j} := \begin{pmatrix} 2^{ j/2 } & 0 \\ 0 & 2^j \end{pmatrix}, %\label{eq:parabolicscalingmatrix}
\end{align*}
and the \emph{shearing matrices} with parameter $k \in \Z$ are 
\begin{align*}
	S_k = \begin{pmatrix} 1 & k \\ 0 & 1 \end{pmatrix}.%\label{eq:shearmatrix2}
\end{align*}
These operations, together with the standard integer shift of functions in $L^2(\R^2)$ are then used to define the \emph{cone adapted discrete shearlet system}.
\begin{deff}[\cite{KitKutLim}]\label{Definition:ShearletSystem}
Let $\phi, \psi, \psitilde \in L^2(\R^2)$ be the \emph{generating functions} and $c=(c_1, c_2) \in (\R^+)^2$. Then the \emph{(cone adapted discrete) shearlet system} is defined as
\begin{align*}
	\mathcal{SH}(\phi,\psi, \psitilde, c) = \Phi(\phi, c_1) \cup \Psi (\psi, c) \cup \widetilde{\Psi}(\psitilde,c),
\end{align*}
where
\begin{align*}
	\Phi(\phi,c_1) = \{ \phi( \cdot - c_1m) \, : \, m \in \Z^2\},
\end{align*}
and
\begin{align*}
	\Psi (\psi, c) &= \left\{ \psi_{j,k,m} = 2^{3j/4 } \psi \left(( S_k A_{2^j}) \cdot - cm \right) \, : \, j \geq 0, |k| \leq 2^{j/2}, m  \in \Z^2 \right\},\\
	\widetilde{\Psi} (\psitilde, c) &= \left\{ \psitilde_{j,k,m} = 2^{3j/4 } \psitilde \left(( S_k^T \widetilde{A}_{2^j}) \cdot - \widetilde{c}m \right) \, : \, j \geq 0, |k| \leq 2^{j/2}, m  \in \Z^2 \right\},
\end{align*}
where the multiplication of $c$ and $\widetilde{c}$ with the translation parameter $m$ should be understood entry wise.
\end{deff}
The shearlet system has a multiscale structure, however, it does not form a classical \emph{multiresolution analysis} (MRA), see \cite{Dau, Mallat} for a definition of an MRA. Moreover, it is still an open question whether there exists compactly supported orthonormal shearlet bases. Nevertheless, shearlets can form a frame.

\begin{theorem}[\cite{KitKutLim}]\label{theorem:shearletsframe}
Let $\phi, \psi \in \L^2(\R^2)$ such that
\begin{align*}
	|\phihat(\xi_1, \xi_2)| \leq C_1 \min \{1, | \xi_1|^{-r} \} \min \{ 1, |\xi_2|^{-r} \}
\end{align*}
and
\begin{align*}
	|\psihat(\xi_1, \xi_2)| \leq C_2 \min \{1, | \xi_1|^{\alpha} \} \min \{1, | \xi_1|^{-r} \} \min \{ 1, |\xi_2|^{-r} \},
\end{align*}
for some constants $C_1,C_2 >0$ and $\alpha > r >3$. Further let $\psitilde (x_1,x_2) = \psi(x_2, x_1)$ and assume there exists a positive constant $A>0$ such that
\begin{align*}
    |\phihat(\xi)|^2 + \sum_{j \geq 0} \sum_{|k| \leq \lceil 2^{j/2}  \rceil} |\psihat(S_k^T (A_j)^{-1} \xi)|^2 + \sum_{j \geq0} \sum_{|k| \leq \lceil 2^{j/2}  \rceil} |\widehat{\psitilde}(S_k (\widetilde{A}_j)^{-1} \xi)|^2 > A
\end{align*}
holds almost everywhere. Then there exists $c = (c_1, c_2) \in (\R^+)^2$ such that the cone-adapted shearlet system $\mathcal{SH}(\phi,\psi, \psitilde, c)$ forms a frame for $L^2(\R^2)$.
\end{theorem}
An explicit construction of compactly supported shearlets that form a frame can be found in \cite{KitKutLim}.  Note that the frequency decay assumptions on the generators given in Theorem \ref{theorem:shearletsframe} are essential and secure that the overlapping in frequency domain is controlled. In particular, the frame constants depend on this behavior, \cite{Lim,KitKutLim,KutLabBook}. For more information about shearlets we refer to \cite{ShearletBook} and the references therein.

\subsection{Sparse approximation of cartoon-like functions}\label{subsec:ApproxCartoon}

Shearlets are vastly used in applied harmonic analysis due to the \emph{optimal sparse approximation rate for cartoon-like functions}. We now recall the definition of cartoon-like functions and the results about optimal sparse approximation rate.

The set of cartoon-like functions was first introduced in \cite{DonSparse} and since then has been frequently used as a mathematical model for natural images.

\begin{deff}
Let $\nu >0$ and $f : \R^2 \longrightarrow \C$ a function of the form
\begin{align*}
	f = g + h \chi_B,
\end{align*}
where $B \subset [0,1]^2$ is a set whose boundary is a closed $C^2$ curve with curvature bounded by $\nu$ and $g,h \in C^2(\R)$ are compactly supported in $[0,1]^2$ with $\| g\|_{C^2}, \| h \|_{C^2} \leq 1$. Then we call $f$ a \emph{cartoon-like function}. The set of cartoon-like functions is called \emph{class of cartoon-like functions} and is denoted by $\mathcal{E}^2(\nu)$.
\end{deff}

The following theorem shows the sparse approximation rate for shearlets.
\begin{theorem}[\cite{KutLim}]\label{theorem:OptimalSparseApprox}
Let $\mathcal{SH}(\phi, \psi, \psitilde, c)$ be a compactly supported shearlet frame with
\begin{itemize}
\item[i)] $| \psihat(\xi)| \leq C \min \{1, | \xi_1|^{\alpha} \} \min \{1, | \xi_1|^{-r} \} \min \{ 1, |\xi_2|^{-r} \}$, 
\item[ii)] $\left| \frac{\partial}{\partial \xi_2} \psihat(\xi)\right| \leq |h(\xi_1)|\left ( 1+ \frac{\xi_2}{\xi_1} \right)^{-r}$,
\end{itemize}
where $\alpha >5, r \geq 4, h \in L^1(\R), C>0$ is some constant. Further assume $\psitilde$ fulfills i) and ii) with switched roles of $\xi_1$ and $\xi_2$. Then for any fixed $\nu>0$ and $f \in \mathcal{E}^2(\nu)$ we have
\begin{align*}
	\| f - f_N \|_2^2 \leq C' (\log N)^3 N^{-2}, \quad \text{ as } N \longrightarrow \infty,
\end{align*}
where $f_N$ is the best $N$-term approximation obtained using the $N$ largest shearlets coefficients of $f$ in magnitude allowing polynomial depth search
and $C'$ is some constant independent of $N$.
\end{theorem}

\begin{rem}
The rate obtained in Theorem \ref{theorem:OptimalSparseApprox} is also achieved in \cite{GuoLab} for band-limited shearlet systems and for curvelets in \cite{CanDon}.
\end{rem}

In \cite{DonSparse} it was shown, that under some assumptions on the representation system and the selection procedure of the largest coefficients in magnitude, the optimal approximation rate of a cartoon-like function $f$ obeys
\begin{align*}
	\| f - f_N \|^2 \leq C N^{-2}.
\end{align*}
This rate is, up to a logarithmic factor, achieved by shearlets, cf. Theorem \ref{theorem:OptimalSparseApprox}. Moreover, compared to the $N^{-2}$ factor, the logarithmic term can be neglected and, thus, the rate in Theorem \ref{theorem:OptimalSparseApprox} is referred to as almost {optimal sparse approximation rate}.

\begin{rem}
It is commonly known, that standard wavelet orthonormal bases only obtain a rate of order $N^{-1}$, see \cite{KutLabBook}. Moreover both rates for shearlets and wavelets are sharp, that is, there exists cartoon-like functions such that the rate is achieved.
\end{rem}

\subsection{Assumptions on the generator}\label{subsec:AssumptionsOnGenerators}
For the rest of this paper we assume $\phi$ and $\psi$ to be compactly supported functions in $L^2(\R^2)$ and define $\psitilde(x_1,x_2) := \psi(x_2,x_1)$. 
Moreover, we assume the shearlets to have sufficient vanishing moments and frequency decay, more precisely, we assume there exist some constants $C_1, C_2>0$ such that
\begin{align}
	| \phihat(\xi_1,\xi_2) | \leq C_1\cdot  \frac{1}{(1 + | \xi_1|)^{r}}  \frac{1}{(1 + | \xi_2|)^{r}}, \label{eq:DecayPhi}
\end{align}
and 
\begin{align}
	|\psihat(\xi_1,\xi_2)| \leq C_2 \cdot \min \{1, |\xi_1|^\alpha\} \cdot \frac{1}{(1+ |\xi_1|)^{r}}\frac{1}{(1+ | \xi_2|)^{r}}. \label{eq:DecayPsi}
\end{align} 
where the regularity parameters $\alpha, r >0$ are large enough so that the shearlet system forms a frame for $L^2(\R^2)$. These assumptions are equivalent to the frequency decay assumptions stated in Theorem \ref{theorem:shearletsframe}.

\section{Stable shearlet reconstructions from Fourier measurements}\label{section:SampRecSpace}

We will now define the reconstruction space and the sampling space for which we determine a stable sampling rate.

\subsection{Shearlet reconstruction space}\label{subsection:SampRecSpace}

Without loss of generality we can assume that the generating scaling function $\phi$ and the shearlets $\psi, \psitilde$ are compactly supported in $[0,a]^2$ where $a$ is some positive integer. Then we consider all scaling functions whose support intersect $[0,a]^2$ and denote this index set by $\Omega$, i.e.
\[
	\Omega = \left \{ m \in \Z^2 \, : \, \suppp \phi_{m} \cap [0,a]^2 \neq \emptyset \right \} = \{ (m_1,m_2) \in \Z^2 \, : \, -a \leq m_1, m_2 \leq a \}.
\]
Similarly, we consider all shearlets whose support intersect $[0,a]^2$. For this, let $J-1 \in \N \cup \{ 0 \}$ be a fixed scale. Then $\Lambda_J$ denotes the following paramater set
\[ 
	\Lambda_J = \left \{ (j,k,m) \in \Z \times \Z \times \Z^2 \, : \, 0 \leq j \leq J -1, |k|\leq 2^{j/2}, m \in \Omega_{j,k} \right\},
\]
where $\Omega_{j,k} = \{ m \in \Z^2 \, : \, \suppp \psi_{j,k,m} \cap [0,a]^2 \neq \emptyset \}$ is, due to the compact support of each $(\psi_{j,k,m})_m$, of finite cardinality and for $\psitilde$ and the second cone we analogously write
\[ 
	\widetilde{\Lambda}_J = \left \{ (\widetilde{j},\widetilde{k},\widetilde{m}) \in \Z \times \Z \times \Z^2 \, : \, 0 \leq \widetilde{j} \leq J -1, |\widetilde{k}|\leq 2^{j/2}, \widetilde{m} \in \widetilde{\Omega}_{\widetilde{j},\widetilde{k}}\right\}
\]
with $\widetilde{\Omega}_{\widetilde{j},\widetilde{k}} = \{ \widetilde{m} \in \Z^2 \, : \, \suppp \psitilde_{\widetilde{j},\widetilde{k},\widetilde{m}} \cap [0,a]^2 \neq \emptyset  \}$ being of finite cardinality, respectively. In fact, the scale determines the number of translation and will be of order $2^{3/2j}$ for a fixed scale $j$.
The reconstruction space at scale $J-1$ is then defined as
\begin{align}
	\Rcal_{N_J} :=  \spann & \bigg\{\left\{ \phi_{m} \, : \, m \in \Omega \right\}  
	\cup\left\{ \psi_{j,k,m} \, : \, (j,k,m) \in \Lambda_J \right\}  
	  \cup \left\{ \psitilde_{\widetilde{j},\widetilde{k},\widetilde{m}} \, : \, (\widetilde{j},\widetilde{k},\widetilde{m}) \in \widetilde{\Lambda}_J  \right \}\bigg\}.\label{ShearReconSpace2}
\end{align}	
Up to a fixed scale $J-1$ we, asymptotically, have $N_J = 2^{2J}$ many generating functions  in $\Rcal_{N_J}$ as $J \to \infty$. %, see, e.g.  \cite{KutLimZhu}. 
Note that by construction at each scale we only have finitely many elements, therefore, an ordering can be performed quite naturally namely we order the system  along scales and within the scales, the shearing is ordered from $-2^{j/2}$ to $2^{j/2}$ and the translations in $\Z^2$ are ordered in a lexicographical manner for each fixed shearing and scaling parameter. 

By $\Rcal$ we denote the reconstruction space that contains all shearlets across all scales, i.e.
\begin{align*}
\Rcal = \overline{\spann} \bigg\{&  \left\{ \phi_{m} \, : \, m \in \Omega \right\}  \cup\left\{ \psi_{j,k,m} \, : \, (j,k,m) \in \Lambda_J , J \in \N \cup \{ 0 \} \right\} \\
& \qquad  \cup \left\{ \psitilde_{\widetilde{j},\widetilde{k},\widetilde{m}} \, : \, (\widetilde{j},\widetilde{k},\widetilde{m}) \in \widetilde{\Lambda}_J , J \in \N  \cup \{ 0 \} \right \}\bigg\}.
\end{align*}

\subsection{Fourier sampling space}

To define the sampling space we first choose $T_1,T_2>0$ sufficiently large such that
\begin{align*}
	\Rcal \subset L^2([-T_1,T_2]^2).
\end{align*}
Let $\varepsilon \leq \frac{1}{T_1+T_2} <1$ control the sampling density. Then we define the sampling vectors on the uniform grid $\Z^2$ by
\begin{align}
	s_\ell^{(\varepsilon)} = \varepsilon e^{2 \pi i \varepsilon \langle \ell, \cdot \rangle} \cdot \chi_{[-T_1,T_2]^2}, \quad \ell \in \Z^2. \label{s_l}
\end{align}
Based on these sampling vectors we define the sampling space $\Scal$ by
\begin{align*}
	\Scal^{(\varepsilon)}=  \overline{\spann} \left\{ s_\ell^{(\varepsilon)} \, : \, \ell \in \Z^2 \right\}.
\end{align*}
For $M = (M_1,M_2 )\in \N \times \N$ let
\begin{align*}
	\Scal^{(\varepsilon)}_{M} = \spann \left\{ s_\ell^{(\varepsilon)} \, : \, \ell = (\ell_1,\ell_2) \in \Z^2, - M_i \leq \ell_i \leq M_i, i =1,2 \right\}
\end{align*}
be the finite dimensional sampling space. Note that $M=(M_1,M_2)$ determines the size of the measured grid and the total number of possible measurements are in this case asymptotically of order $M_1\cdot M_2$. 

The task for us is now to describe the relationship between $M$ and $N$ in terms of the stable sampling rate which means such that stable and convergent reconstructions exist. In the event that the stable sampling rate is linear, i.e. $\Theta(N, \theta) = \mathcal{O}(N)$ we would  have $M_1 \cdot M_2 = \mathcal{O}(2^{2J})$ as for the wavelet case, \cite{AHP1, AHKM}.

\subsection{Main results}\label{section:MainResult}

Our main result shows that the angle between the shearlet reconstruction space and the Fourier sampling space can be controlled with an almost linear stable sampling rate. 
\begin{theorem}\label{maintheoremshearlets}
Let $\Scal\Hil(\phi, \psi, \psitilde)$ be a compactly supported shearlet frame with generators $\phi, \psi,$ and $\psitilde$ and let $N \leq N_J = \mathcal{O}(2^{2J})$. Then for all $\theta >1$ there exists $S_{\theta}>0$ such that
\begin{align*}
	c_{N,M} = \inf_{\substack{ f \in \Rcal_{N} \\ \| f \| = 1}} \| P_{\Scal_M^{(\varepsilon)}} f \| \geq \frac{1}{\theta}, %\quad \text{as } N \to \infty
\end{align*}
where $M = (M_1, M_2) \in \N \times \N$ with $M_i = \lceil S_{\theta} A_N^{-1/(2r-1)}\cdot 2^{J(1+\delta)}/\varepsilon\rceil, \delta \geq \frac{2}{2r-1}$ and $r>0$ is the regularity parameter from (\ref{eq:DecayPhi}) and (\ref{eq:DecayPsi}). Therefore $\Theta(N, \theta) = \mathcal{O}(N^{1+\delta}A_N^{-1/(2r-1)})$. Further, the constant $S_{\theta}$ does not depend on $N$ but on $\theta, \alpha,$ and $r$.
\end{theorem}

\begin{rem}
The rate appearing in Theorem \ref{maintheoremshearlets} is slightly worse than the rate for wavelets \cite{AHP1, AHKM} where the stable sampling rate is indeed linear, i.e. $\Theta(N, \theta) = \mathcal{O}(N)$. More precisely, one has to oversample by a factor $N^{\delta}$ and the rate depends on the lower frame bound $A_N$ of the finite shearlet system that is used to span $\Rcal_N$.
\end{rem}

\begin{rem}
If we interpret the reciprocal of the lower frame bound as a measure for redundancy or stability, then our result states that more samples are necessary if the system is more redundant or less stable. Note that if the frame is tight, then the additional factor $A_N^{-1}$ can be neglected.
\end{rem}

If one uses the same number of shearlets and wavelets to reconstruct a function from its Fourier measurements, then by Theorem \ref{maintheoremshearlets}, one has to acquire more Fourier measurements to guarantee the existence of the GS solution if shearlets are used compared to wavelets. This immediately raises the question why the use of shearlets in the context of GS is justified, if one has to acquire more samples to guarantee stable and convergent solutions. We address this issue from a different perspective. Suppose we were to sample an image using the rate computed in Theorem \ref{maintheoremshearlets} for shearlets, but we want to reconstruct the image from the same number of measurements using shearlets and wavelets. Then, due to the stable sampling rate for shearlets and wavelets, we are allowed to use more wavelets than shearlets. More precisely, let $\sigma : \N \longrightarrow \N$ be the following oversampling function
\begin{align}
	\sigma(N) = \left \lceil N^{1+\delta} \frac{1}{A_N^{\frac{2}{2r-1}}} \right \rceil, \quad N \in \N. \label{eq:sigmafunc}
\end{align}
Then the following result holds.

\begin{cor}\label{cor:ErrorEstimates}
Let $\Scal\Hil(\phi, \psi, \psitilde)$ be a compactly supported shearlet frame with generators $\phi, \psi,$ and $\psitilde$ with sufficiently large regularity $r$ and $\delta \geq 2/(2r-1)$. Further, let $f$ be a cartoon-like function $f$. For $N \in \N$ denote by  $f^s_N$ the best $N$-term approximation of $f$ using shearlets and $f^w_{\sigma(N)}$ the best $\sigma(N)$-approximation using wavelets, cf. Subsection \ref{subsec:ApproxCartoon}. Further, let $N_J = 2^{2J}, J \in \N$ be the smallest number such that $N, \sigma(N) \lesssim N_J$. If $\Rcal^s_{N_J}$ denotes the shearlet reconstruction space space so that $f_N^s \in \Rcal_{N_J}^s$ and $\Rcal_{\sigma({N_J})}^w$ denotes the wavelet reconstruction space so that $f_{\sigma(N)}^w \in \Rcal_{\sigma({N_J})}^w$,
then
\begin{align*}
	\| f - G_{{N_J},M}^s(f) \|  \lesssim N^{-1} (\log  N)^{3/2},
\end{align*}
and
\begin{align*}
	\| f - G_{\sigma({N_J}),M'}^w(f) \| &\lesssim \sigma(N)^{-1/2},
\end{align*}
where $M$ and $M'$ are of order $\sigma(N_J)$  and $G_{{N_J},M}^s(f)$ and $G_{{N_J}M}^w(f)$ are the GS solution with respect to the spaces $\Rcal_{N_J}^s$ and $\Rcal_{\sigma({N_J})}^w$. %Moreover, there exists a cartoon-like function, so that the above inequalities hold with equality up to a constant.
\end{cor}

\begin{rem}
Note that $N^{-1} (\log  N)^{3/2} \lesssim \sigma(N)^{-1/2}$ if
\begin{align}
	N^{-(1-\delta)/2}(\log N)^{3/2} \lesssim A_N^{1/(2r-1)}, \label{eq:asymptotics}
\end{align}
and therefore, if $ \|f - f_N^s \| \asymp \| f - G_{{N_J},M}^s(f) \|$ and $ \|f - f_N^w \| \asymp \| f - G_{\sigma({N_J}),M}^w(f) \|$, then
\begin{align*}
	\| f - G_{{N_J},M}^s(f) \| \lesssim 	\| f - G_{\sigma({N_J}),M'}^w(f) \|,
\end{align*}
where "$\asymp$" means equality up to some constant.
\end{rem}

In general, there is no control about the lower frame bound of the finite shearlet system known, thus, the reader might wonder whether \eqref{eq:asymptotics} can hold. In Figure \ref{fig:asymptotics} in Section \ref{section:Numerics} we numerically confirm the validity of such an behavior of the lower frame bound.

\section{Proofs}\label{section:Proofs}

We first sketch the idea of the proof of Theorem \ref{maintheoremshearlets}.

\subsection{Intuitive argument for the proof of Theorem \ref{maintheoremshearlets}}

In order to prove the main theorem we have to bound
\begin{align*}
	\inf_{\substack{f \in \Rcal_N \\ \| f \| =1}} \| P_{\Scal_M^{(\eps)}} f \|
\end{align*}
from below for respective $M$ and $N$. For any $f \in \Rcal_N$  we have
\begin{align*}
	\| P_{\Scal_M^\eps} f \|^2 = \sum_{ \ell \in I_M} | \langle f, s^{(\eps)}_\ell \rangle |^2,
\end{align*}
where $I_M = \{ \ell= (\ell_1, \ell_2) \in \Z^2 \, : \, -M_i \leq \ell_i \leq M_i, i = 1,2 \}$. Hence, we could equivalently bound
\begin{align}
	\| P^\perp_{\Scal_M^{(\eps)}} f \|^2 = \sum_{ \ell \in (I_M)^c} | \langle f, s^{(\eps)}_\ell \rangle |^2, \label{eq:Idea1}
\end{align}
from above. Now, the key idea is to use the effective frequency support of the shearlets, in particular, the energy of  $(\psi_{j,k,m})^\wedge$ is essentially localized in frequency bands of width $2^{2j}$ up to some constant, see Subsection \ref{subsec:EffFreqSupp} for precise statements. Moreover, since $f \in \Rcal_N$ is a linear combination of shearlets up to a fixed scale $J-1$ the function $f$ might be essentially supported in $[-C2^J, C 2^J]^2$ for some constant $C$ as well. Hence, by making the grid $I_M$ large and dense enough the term in \eqref{eq:Idea1} should be small. However, there are two main concerns: First, the effective support could grow faster than linearly with respect to scaling factor $2^j$. Second, taking linear combinations of shearlets might destroy the control of the frequency behavior and hence the essential supports in Fourier domain. Unfortunately, both obstacles are inevitable causing the additional factor $N^{\delta}$ and the dependence on the lower frame bound $A_N$ in the stable sampling rate, respectively. More precisely, if $f \in \Rcal_N$ then there exists $c_1, \ldots, c_N \in \C$ such that
\begin{align*}
	f = \sum_{\lambda \leq N} c_\lambda \psi_\lambda
\end{align*}
but in general we have no control over the coefficients $(c_\lambda)_{\lambda \leq N}$. Therefore in order to resolve the latter concern, we expand $f$ as follows:
$$
	f = \sum_{\lambda \leq N } \langle f,   S_N^{-1}\psi_\lambda \rangle \psi_\lambda.
$$ 
Then
$$
	\sum_{\lambda \leq N } |\langle f,   S_N^{-1} \psi_\lambda \rangle |^2 \leq \frac{1}{A_N} \| f \|^2.
$$ 
This is how the lower frame bound of the finite shearlet system comes into the proof and thus effects the stable sampling rate. Note that this problem does not appear if the reconstruction system is an orthonormal basis or a Riesz basis, thus, an additional penalty of the stable sampling rate is not necessary for wavelets.

%Next we discuss the effective frequency support of the shearlets $(\psi_\lambda)_\lambda$.

\subsection{Effective frequency support}\label{subsec:EffFreqSupp}

The decay assumptions (\ref{eq:DecayPhi}) and  (\ref{eq:DecayPsi}) give rise to the essential support of each shearlet atom $\psi_{\lambda}$ in frequency, see Figure \ref{fig:EssSupp}. The next result secures that the effective frequency support is controllable.

\begin{prop}\label{prop:EffFreqSupp}
Let $J \in \N$ and $(\psi_\lambda)_{\lambda \leq N_J}$ be all shearlets up to scale $J-1$ and $\omega >0$. Then there exists a constant $S := S(\omega, r, \eps)$ such that for $I_M := \{ (\ell_1, \ell_2) \in \Z^2 \, : \, -M_i \leq  \ell_i \leq M_i, i =1,2 \}$ with  $M_i = S2^{J(1+\delta)}, i=1,2$ and $\delta \geq \frac{2}{2r-1}$ we have
\begin{align*}
	\sum_{\ell \in (I_M)^c} \sum_{\lambda \leq N_J} |(\psi_\lambda)^\wedge(\eps \ell)|^2 \leq \omega.
\end{align*}
\end{prop}
\begin{figure}[H]
\centering
	\includegraphics[scale=0.45]{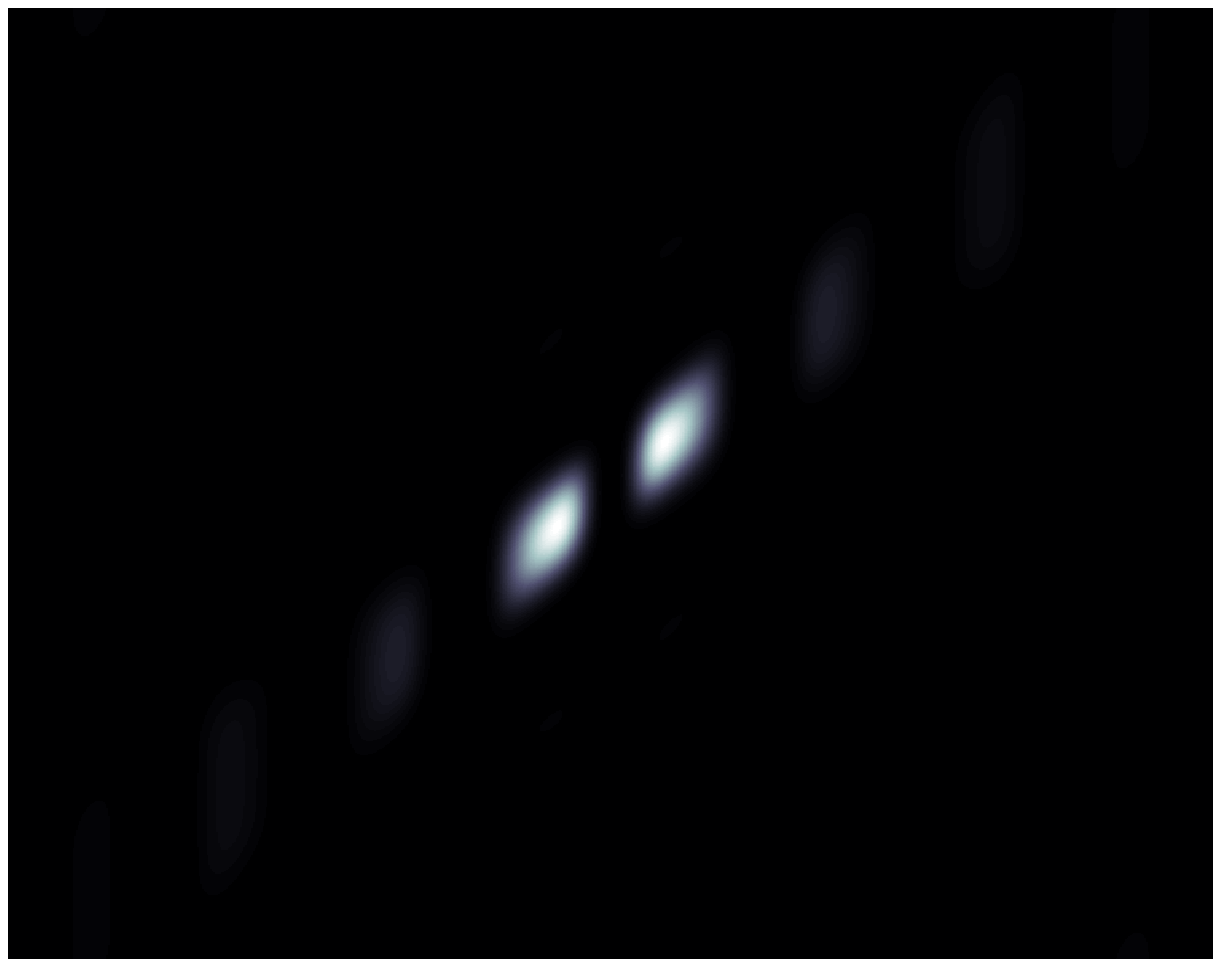}
	\includegraphics[scale=0.45]{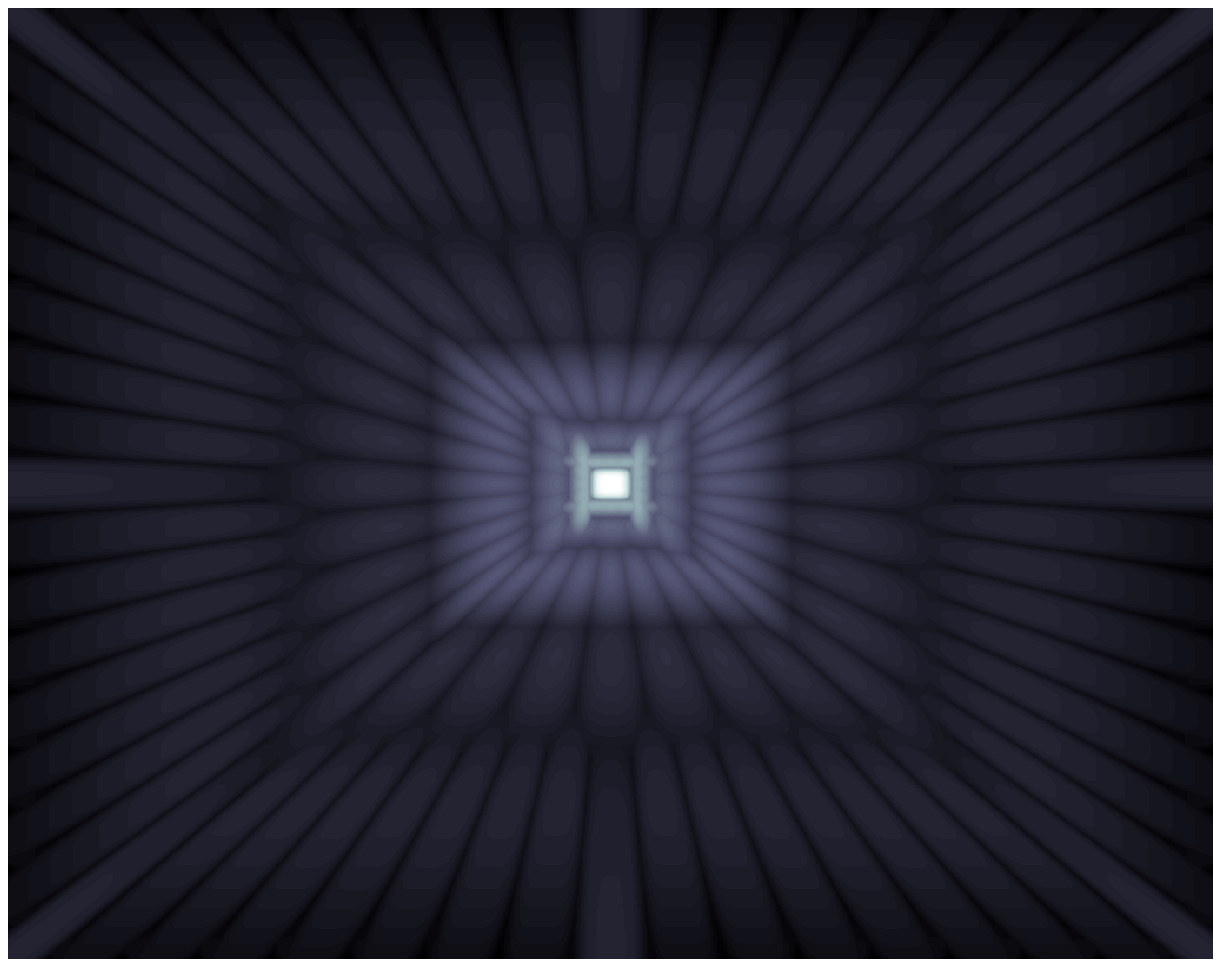}
\caption{Effective frequency support of shearlets and tiling of the frequency plane, light regions correspond to large values, therefore regions with large energy and dark regions to small values, i.e. regions with small energy. Images are computed using the ShearLab package downloaded from \texttt{http://www.shearlab.org/}} \label{fig:EssSupp}
\end{figure}

To prove Proposition \ref{prop:EffFreqSupp} we will use the following lemma which can be found in  \cite{Grafakos} (Appendix K).

\begin{lemma}[\cite{Grafakos}]\label{lemma:Grafakos}
For $y \in \R , r>1,$ and $a,b >0$ we have
\begin{align*}
	\int_0^\infty \frac{1}{(1 + a|x|)^r} \frac{1}{(1 + b|x-y|)^r} \, dx \lesssim \frac{1}{\max(a,b)}  \frac{1}{(1 +  \min ( a,b)|y|)^r}.
\end{align*}
\end{lemma}

\begin{proof}[Proof of Proposition \ref{prop:EffFreqSupp}]
Let $\omega >0$. Then we have to show the existence of an $M = (M_1,M_2) \in \N \times \N$, so that
\begin{align}
\left( \sum \limits_{ |l_2| > M_2}\sum \limits_{ |l_1| > M_1}  + \sum \limits_{ |l_2| < M_2}\sum \limits_{ |l_1| > M_1}  + \sum \limits_{ |l_2| > M_2}\sum \limits_{ |l_1| < M_1} \right) \sum_{\lambda \leq N_J} |(\psi_\lambda)^\wedge(\eps \ell)|^2
	\leq \omega, \label{eq:ProofLemmaFreq1}
\end{align}
where $M_1,M_2$ scales like $S 2^{J(1+\delta)}$ with a constant $S$ independent of $J$ sufficiently large. Note that
\begin{align}
	\sum_{\lambda \leq N_J}& |(\psi_\lambda)^\wedge(\eps \ell)|^2  \nonumber \\
	& =  \sum \limits_{m' \in \Omega} \left | \varepsilon e^{-2\pi i \varepsilon \langle l, m' \rangle} \phihat(\varepsilon l) \right|^2 + \sum \limits_{(j,k,m)\in \Lambda_{J} } \left|  \frac{\varepsilon}{2^{3j/4}} e^{-2 \pi i  \varepsilon \langle  \left(S_{k} A_{2^j}\right)^{-T} l, m \rangle} \psihat \left(\varepsilon  \left(S_{k} A_{2^j}\right)^{-T} l \right) \right|^2+ \nonumber \\
	&\makebox[1cm][c]{}+  \sum \limits_{(\widetilde{j},\widetilde{k}, \widetilde{m}) \in \widetilde{\Lambda}_{J}}  \left |  \frac{\varepsilon}{2^{3\widetilde{j}/4}} e^{-2 \pi i  \varepsilon \langle \left(  S_{\widetilde{k}} \widetilde{A}_{2^{\widetilde{j}}}\right)^{-T}  l, \widetilde{m} \rangle} \widehat{\psitilde}\left(\varepsilon \left(  S_{\widetilde{k}} \widetilde{A}_{2^{\widetilde{j}}}\right)^{-T} l \right) \right|^2, \nonumber \\
 \nonumber\\
	&\leq C  \left| \varepsilon \phihat(\varepsilon l)\right|^2 +  \sum \limits_{ j=0}^{J-1} \sum \limits_{ k = - 2^{j/2}}^{ 2^{j/2}} \left| \psihat \left(\varepsilon  \left(S_{k} A_{2^j}\right)^{-T} l \right)\right|^2+   \sum \limits_{\widetilde{j} =0}^{J-1} \sum \limits_{ \widetilde{k} = - 2^{\widetilde{j}/2}}^{2^{\widetilde{j}/2}}  \left|   \widehat{\psitilde}\left(\varepsilon \left(  S_{\widetilde{k}} \widetilde{A}_{2^{\widetilde{j}}}\right)^{-T} l \right)\right|^2 ,
\end{align}
since the cardinality of $\Omega_{j,k}$ and $\widetilde{\Omega}_{\widetilde{j},\widetilde{k}}$ is of order $2^{3/2j}$.

Denote by $\mathrm{I}, \mathrm{II}, \mathrm{III}$ the following terms
\begin{align*}	
	\text{\textrm{I}} & :=  \left( \sum \limits_{ |l_2| > M_2}\sum \limits_{ |l_1| > M_1}  + \sum \limits_{ |l_2| < M_2}\sum \limits_{ |l_1| > M_1}  + \sum \limits_{ |l_2| > M_2}\sum \limits_{ |l_1| < M_1} \right) \left|\varepsilon \phihat(\varepsilon l)\right|^2,\\
	\text{\textrm{II}} & := \left( \sum \limits_{ |l_2| > M_2}\sum \limits_{ |l_1| > M_1}  + \sum \limits_{ |l_2| < M_2}\sum \limits_{ |l_1| > M_1}  + \sum \limits_{ |l_2| > M_2}\sum \limits_{ |l_1| < M_1} \right) \sum \limits_{ j=0}^{J-1} \sum \limits_{ k = - 2^{j/2}}^{ 2^{j/2}}\left| \psihat \left(\varepsilon  \left(S_{k} A_{2^j}\right)^{-T} l \right)\right|^2, \\
	\text{\textrm{III}} & :=  \left( \sum \limits_{ |l_2| > M_2}\sum \limits_{ |l_1| > M_1}  + \sum \limits_{ |l_2| < M_2}\sum \limits_{ |l_1| > M_1}  + \sum \limits_{ |l_2| > M_2}\sum \limits_{ |l_1| < M_1} \right) \sum \limits_{\widetilde{j} =0}^{J-1} \sum \limits_{ \widetilde{k} = - 2^{\widetilde{j}/2}}^{2^{\widetilde{j}/2}} \left| \widehat{\psitilde}\left(\varepsilon \left(  S_{\widetilde{k}} \widetilde{A}_{2^{\widetilde{j}}}\right)^{-T} l \right) \right|^2 . 
\end{align*}
We will estimate each sum by using the decay conditions (\ref{eq:DecayPhi}) and  (\ref{eq:DecayPsi}), respectively, in order to obtain \eqref{eq:ProofLemmaFreq1} for sufficiently large $S$ independent on $J$ with $M_i = \frac{2^{J(1+\delta)}}{\varepsilon} S , i =1,2$. 
Since
\begin{align}
	\left|\varepsilon \phihat(\varepsilon l)\right|^2 & \leq   C_1^2 \varepsilon^2 \left| \frac{1}{(1 + |\varepsilon l_1|)^r}  \frac{1}{(1 + |\varepsilon l_2|)^{r}}  \right|^2.  \label{eq:ProofLemmaFreq2}
\end{align}
we immediately have
\begin{align}
	\sum \limits_{|l_2|>M_2}  \sum \limits_{|l_1|<M_1}   \left|\phihat(\varepsilon l)\right|^2   
	& \leq \sum \limits_{|l_2|>M_2} \sum \limits_{|l_1|<M_1} C_1^2  \frac{1}{(1 + |\varepsilon l_1|)^{2r}}  \frac{1}{(1 + |\varepsilon l_2|)^{2r}} \nonumber \\
	& \leq C_1^2  \sum \limits_{|l_1|>M_1} \frac{1}{(1 + |\varepsilon l_1|)^{2r}}   \nonumber \\
	& \leq C_1^2  \frac{1}{(1+S^{J(1+ \delta)})^{2r-1}},   \label{eq:ProofLemmaFreq3}
\end{align}
where the constant $C_1$ changed in each step. Analogously, one obtains
\begin{align}
	\sum \limits_{|l_2|<M_2}  \sum \limits_{|l_1|>M_1} \left|\phihat(\varepsilon l)\right|^2   &  \leq  C_1^2  \frac{1}{(1+S 2^{J(1+ \delta)})^{2r-1}}, \nonumber \\
	\sum \limits_{|l_2|>M_2}  \sum \limits_{|l_1|>M_1}  \left|\phihat(\varepsilon l)\right|^2  &   \leq   C_1^2 \frac{1}{(1+S 2^{J(1+ \delta)})^{2r-1}}. \label{eq:ProofLemmaFreq4}
\end{align}
Hence, combinging \eqref{eq:ProofLemmaFreq3} and \eqref{eq:ProofLemmaFreq4} gives
\begin{align}
	\mathrm{I} \leq   C_1^2 \frac{1}{(1+S 2^{J'(1+ \delta)})^{2r-1}} \label{eq:ProofLemmaFreq5}
\end{align}
Regarding \textrm{II} we first have by (\ref{eq:DecayPsi})
\begin{align}
 \sum \limits_{ j=0}^{J-1} \sum \limits_{ k = - 2^{j/2}}^{ 2^{j/2}} \left|\psihat \left(\varepsilon  \left(S_{k} A_{2^j}\right)^{-T}l\right)\right|^2 %\nonumber\\
	 \leq  \sum \limits_{ j=0}^{J-1} \sum \limits_{ k = - 2^{j/2}}^{ 2^{j/2}}  \frac{C_2^2}{(1+|\varepsilon 2^{-j}l_1|)^{2r}(1+ |-\varepsilon k2^{-j}l_1 + \varepsilon 2^{-j/2} l_2|)^{2r}} , \label{eq:ProofLemmaFreq6} 
\end{align}
and, likewise for \textrm{III},
\begin{align}
\sum \limits_{\widetilde{j} =0}^{J-1} \sum \limits_{ \widetilde{k} = - 2^{\widetilde{j}/2}}^{2^{\widetilde{j}/2}} \left|\widehat{\psitilde} \left(\varepsilon  \left(S_{\widetilde{k}}^T A_{2^{\widetilde{j}}}\right)^{-T}l\right)\right|^2 
\leq \sum \limits_{\widetilde{j} =0}^{J-1} \sum \limits_{ \widetilde{k} = - 2^{\widetilde{j}/2}}^{2^{\widetilde{j}/2}}   \frac{C_3^2}{(1+ |\varepsilon 2^{-\widetilde{j}}l_2|)^{2r}(1+ | \varepsilon2^{-\widetilde{j}/2} l_1 - \varepsilon\widetilde{k} 2^{-\widetilde{j}} l_2|)^{2r}}  . \label{eq:ProofLemmaFreq7}
\end{align}
We only continue to estimate  (\ref{eq:ProofLemmaFreq6}) since an estimate for (\ref{eq:ProofLemmaFreq7}) is then obtained analogously.

We consider two cases. The first case concerns shearlets, that are \ti{wavelet-like}. More precisely, we consider shearlets with no shearing first, so parabolically scaled  wavelets. 

\textbf{Case \textrm{I}}: Let  $0 \leq j\leq J-1, k = 0$. By direct computations similar to \eqref{eq:ProofLemmaFreq3} we have
\begin{align}
	\sum \limits_{|l_2| > M_2}&\sum \limits_{|l_1|<M_1}  \sum \limits_{ j=0}^{J-1}C_2^2  \frac{1}{(1+|\varepsilon 2^{-j}l_1|)^{2r}} \frac{1}{(1+ | \varepsilon 2^{-j/2} l_2|)^{2r}}  \nonumber\\
	& \leq  \sum \limits_{ j=0}^{J-1} C_2^2 \sum \limits_{|l_2| > M_2}\frac{1}{(1+ | \varepsilon 2^{-j/2} l_2|)^{2r}} \sum \limits_{|l_1|<M_1} \frac{1}{(1+|\varepsilon 2^{-j} l_1|)^{2r}} \nonumber \\
	& \leq \sum \limits_{ j=0}^{J-1}  C_2^2 \frac{2^{j}}{\eps}   \sum \limits_{|l_2| > M_2}\frac{1}{(1+ | \varepsilon 2^{-j/2} l_2|)^{2r}}  \nonumber \\
	&\leq \sum \limits_{ j=0}^{J-1}  C_2^2 \frac{2^{3/2j}}{\eps^2} \frac{1}{(1+  S 2^{J'(1+\delta)-j/2} )^{2r-1}} \nonumber \\
	&\leq   C_2^2 \frac{2^{3/2J}}{\eps^2} \frac{1}{(1+  S 2^{J(1/2+\delta)} )^{2r-1}}. \label{eq:ProofLemmaFreq8}
\end{align}
where $C_2$ changed over time. In the same manner we obtain
\begin{align}
	\sum \limits_{|l_2| < M_2}\sum \limits_{|l_1|>M_1} &\sum \limits_{ j=0}^{J-1}C_2^2    \frac{1}{(1+|\varepsilon 2^{-j}l_1|)^{2r}} \frac{1}{(1+ | \varepsilon 2^{-j/2} l_2|)^{2r}}  \leq  C_2^2 \frac{2^{3/2J}}{\eps^2}\frac{1}{(1+  S 2^{J\delta} )^{2r-1}},  \label{eq:ProofLemmaFreq9}
\end{align}
and
\begin{align} 
	\sum \limits_{|l_1|>M_1} \sum \limits_{|l_2|>M_2} & \sum \limits_{ j=0}^{J-1}  C_2^2  \frac{1}{(1+|\varepsilon 2^{-j}l_1|)^{2r}} \frac{1}{(1+ | \varepsilon 2^{-j/2} l_2|)^{2r}} \nonumber \\
	&\leq   C_2^2 \frac{2^{3/2J}}{\eps^2} \frac{1}{\left((1+  S 2^{J\delta} )(1+  S 2^{J(1/2+\delta)} )\right)^{2r-1}}. \label{eq:ProofLemmaFreq10}
\end{align}
\textbf{Case 2}: Let  $0 \leq j \leq J-1, k \neq 0$. By using Lemma \ref{lemma:Grafakos}  we have
\begin{align}
\sum \limits_{|l_2| > M_2}&\sum \limits_{|l_1|<M_1}\sum \limits_{ j=0}^{J-1} \sum \limits_{\substack{ k = - 2^{j/2} \\k \neq 0}}^{ 2^{j/2}}  C_2^2  \frac{1}{(1+|\varepsilon 2^{-j}l_1|)^{2r}} \frac{1}{(1+ |-\varepsilon k2^{-j}l_1 + \varepsilon 2^{-j/2} l_2|)^{2r}}  \nonumber\\
&\leq \sum \limits_{ j=0}^{J-1} \sum \limits_{\substack{ k = - 2^{j/2} \\k \neq 0}}^{ 2^{j/2}} \sum \limits_{|l_2| > M_2}  C_2^2  \frac{1}{\varepsilon|k2^{-j}|} \frac{1}{(1+|\varepsilon 2^{-j}| | 2^{j/2}l_2/k|)^{2r}} \nonumber \\
&\leq \sum \limits_{ j=0}^{J-1} \sum \limits_{\substack{ k = - 2^{j/2} \\k \neq 0}}^{ 2^{j/2}}  C_2^2   \frac{2^j}{\varepsilon |k|} \frac{|k|2^{j/2}}{\varepsilon}\frac{1}{(1+|\frac{S}{\eps} 2^{J'(1+\delta)-j/2}/k|)^{2r-1}} \nonumber \\
&= \sum \limits_{ j=0}^{J-1} \sum \limits_{\substack{ k = - 2^{j/2} \\k \neq 0}}^{ 2^{j/2}} C_2^2 \frac{2^{3/2j}}{\eps^2} \frac{1}{(1+|\frac{S}{\eps} 2^{J(1+\delta)-j/2}/k|)^{2r-1}} \nonumber \\
&\leq C_2^2 \frac{2^{2J}}{\eps^2} \frac{1}{(1+|\frac{S}{\eps}2^{J\delta}|)^{2r-1}} . \label{eq:ProofLemmaFreq11}
\end{align}
Since
\begin{align*}
\sum \limits_{|l_2| < M_2} \frac{1}{(1+ |-\varepsilon k2^{-j}l_1 + \varepsilon 2^{-j/2} l_2|)^{2r}} \lesssim \int_{\R} \frac{1}{(1+ |-\varepsilon k2^{-j}l_1 + \varepsilon 2^{-j/2} x|)^{2r}} \, d x \lesssim \frac{2^{j/2}}{\varepsilon}
\end{align*}
we have that
\begin{align}
\sum \limits_{|l_2| < M_2}\sum \limits_{|l_1|>M_1} &\sum \limits_{ j=0}^{J-1} \sum \limits_{\substack{ k = - 2^{j/2} \\k \neq 0}}^{ 2^{j/2}} C_2^2  \frac{1}{(1+|\varepsilon 2^{-j}l_1|)^{2r}} \frac{1}{(1+ |-\varepsilon k2^{-j}l_1 + \varepsilon 2^{-j/2} l_2|)^{2r}}  \nonumber \\
& \leq \sum \limits_{ j=0}^{J-1} \sum \limits_{\substack{ k = - 2^{j/2} \\k \neq 0}}^{ 2^{j/2}} C_2^2 \frac{2^{3/2j}}{\eps^2} \frac{1}{(1+|\frac{S}{\eps} 2^{J(1+\delta)-j}|)^{2r-1}} \nonumber \\
& \leq C_2^2 \frac{2^{2J}}{\eps^2} \frac{1}{(1+|\frac{S}{\eps} 2^{J\delta}|)^{2r-1}}. \label{eq:ProofLemmaFreq12}
\end{align}
Finally, the last sum can be bounded as in (\ref{eq:ProofLemmaFreq11}) by first using Lemma \ref{lemma:Grafakos} again
\begin{align}
\sum \limits_{|l_2| > M_2}\sum \limits_{|l_1|>M_1} &  \sum \limits_{ j=0}^{J-1} \sum \limits_{\substack{ k = - 2^{j/2} \\k \neq 0}}^{ 2^{j/2}} C_2^2  \frac{1}{(1+|\varepsilon 2^{-j}l_1|)^{2r}} \frac{1}{(1+ |-\varepsilon k2^{-j}l_1 + \varepsilon 2^{-j/2} l_2|)^{2r}} \nonumber \\
& \leq  \sum \limits_{ j=0}^{J-1} \sum \limits_{\substack{ k = - 2^{j/2} \\k \neq 0}}^{ 2^{j/2}} C_2^2 \frac{2^{3/2j}}{\eps^2} \frac{1}{(1+|S_\theta2^{J(1+\delta)-j}|)^{2r-1}} \nonumber \\
& \leq C_2^2 \frac{2^{2J}}{\eps^2} \frac{1}{(1+|\frac{S}{\eps} 2^{J\delta}|)^{2r-1}}. \label{eq:ProofLemmaFreq13}
\end{align}
hence, combining \eqref{eq:ProofLemmaFreq8}, \eqref{eq:ProofLemmaFreq9}, \eqref{eq:ProofLemmaFreq10}, \eqref{eq:ProofLemmaFreq11} \eqref{eq:ProofLemmaFreq12}, \eqref{eq:ProofLemmaFreq13} yields
\begin{align}
\mathrm{II} \leq C_2^2 \frac{2^{2J}}{\eps^2} \frac{1}{(1+|\frac{S}{\eps} 2^{J\delta}|)^{2r-1}}. \label{eq:ProofLemmaFreq14}
\end{align}
Regarding \textrm{III}, we obtain a similar estimate by performing the same computations as for $\mathrm{II}$, therefore
\begin{align}
\mathrm{III} \leq C_2^2 \frac{2^{2J}}{\eps^2} \frac{1}{(1+|\frac{S}{\eps} 2^{J\delta}|)^{2r-1}}. \label{eq:ProofLemmaFreq15}
\end{align}
Combining \eqref{eq:ProofLemmaFreq5}, \eqref{eq:ProofLemmaFreq14}, \eqref{eq:ProofLemmaFreq15} yields
\begin{align*}
	\mathrm{I} + \mathrm{II} + \mathrm{III} \leq C \frac{\varepsilon^{2r-3}}{(S2^{J(\delta-2/(2r-1)})^{2r-1}},
\end{align*}
with a constant that does not depend on $J$. Therefore, if $\delta \geq \frac{2}{2r-1}$ the result follows for $S$ greater than $C^{-1}(\frac{\varepsilon^{2r-3}}{\omega})^{1/(2r-1)}$.
\end{proof}

\subsection{Proof of Theorem \ref{maintheoremshearlets}}
The  proof of Theorem \ref{maintheoremshearlets} is a consequence of Proposition \ref{prop:EffFreqSupp}. 
\begin{proof}[Proof of Theorem \ref{maintheoremshearlets}]
Let $\theta >1$. Then we want to show
\begin{align}
	\inf_{\substack{ f \in \Rcal_N \\ \| f \| =1}} \| P_{\Scal_M^{(\eps)}} f \| \geq  \frac{1}{\theta} \label{claim:ssrasymp1}
\end{align}
for appropriate $M$. For this, let $f \in \Rcal_N$ with $\| f \| =1$. We prove
\begin{align}
	\| P_{\Scal_M^{(\eps)}}^\perp  f \|^2 \leq  \frac{\theta^2-1}{ \theta^2} , \label{claim:1}
\end{align}
for the claimed $M$. Since $(s_\ell)_\ell$ is an orthonormal system, we have
\begin{align*}
	\| P_{\Scal_M^{(\eps)}}^\perp f \|^2 &=   \sum \limits_{l \in (I_M)^c} |\langle f , s_l^{(\eps)} \rangle |^2 ,
\end{align*}
where $(I_M)^c$ denotes the set complement of $I_M$ in $\Z^2$, in particular
\begin{align*}
	\sum \limits_{l \in (I_M)^c}& |\langle f , s_l^{(\eps)} \rangle |^2 
	 = \sum \limits_{ |l_2| > M_2}\sum \limits_{ |l_1| > M_1} |\langle f, s_l^{(\eps)} \rangle |^2 + \sum \limits_{ |l_2| < M_2}\sum \limits_{ |l_1| > M_1} |\langle f, s_l^{(\eps)} \rangle |^2 + \sum \limits_{ |l_2| > M_2}\sum \limits_{ |l_1| < M_1} |\langle f, s_l^{(\eps)} \rangle |^2.
\end{align*}
Therefore we have
\begin{align*}
	f= \sum_{\lambda \leq {N_{J}}} \langle f,S_N^{-1} \psi_\lambda \rangle \psi_\lambda 
	=  \sum \limits_{m' \in \Omega} \alpha_{m'} \phi_{m'} + \sum \limits_{(j,k,m)\in \Lambda_{J} } \beta_{j,k,m}\psi_{j,k,m} + \sum \limits_{(\widetilde{j},\widetilde{k}, \widetilde{m}) \in \widetilde{\Lambda}_{J}}  \gamma_{\widetilde{j},\widetilde{k},\widetilde{m}} \psitilde_{\widetilde{j},\widetilde{k},\widetilde{m}}.
\end{align*}
with
\begin{align*}
	\alpha_{m'} = \langle \phi_{m'}, s_l \rangle, \quad
	\beta_{j,k,m} = \langle \psi_{j,k,m}, s_l \rangle, \quad
	\gamma_{j,k,m} = \langle \psitilde_{\widetilde{j},\widetilde{k},\widetilde{m}}, s_l \rangle.
\end{align*}
Furthermore,
\begin{align*}
 \sum \limits_{m' \in \Omega} \left|\alpha_{m'} \right|^2  + \sum \limits_{(j,k,m)\in \Lambda_{J} } \left|\beta_{j,k,m}\right|^2 +  \sum \limits_{(\widetilde{j},\widetilde{k}, \widetilde{m}) \in \widetilde{\Lambda}_{J}}  \left|\gamma_{\widetilde{j},\widetilde{k},\widetilde{m}} \right|^2 = \sum_{\lambda \leq {N_{J}}} |\langle S_N^{-1} f, \psi_\lambda \rangle |^2 
 \leq \frac{1}{A_N} \| f \| ^2.
\end{align*}
Therefore by Cauchy Schwarz
\begin{align*}
	\sum \limits_{l \in (I_M)^c}&  |\langle f, s_l \rangle |^2 \nonumber \\
	& = \sum \limits_{l \in (I_M)^c} \bigg|\langle \sum \limits_{m' \in \Omega} \alpha_{m'} \phi_{m'} + \sum \limits_{(j,k,m)\in \Lambda_{J} } \beta_{j,k,m}\psi_{j,k,m} + \sum \limits_{(\widetilde{j},\widetilde{k}, \widetilde{m}) \in \widetilde{\Lambda}_{J}}  \gamma_{\widetilde{j},\widetilde{k},\widetilde{m}} \psitilde_{\widetilde{j},\widetilde{k},\widetilde{m}}, s_l \rangle \bigg|^2 \nonumber \\
	& = \sum \limits_{l \in (I_M)^c} \bigg| \sum \limits_{m' \in \Omega} \alpha_{m'} \widehat{\phi_{m'}}(\varepsilon l) + \sum \limits_{(j,k,m)\in \Lambda_{J} } \beta_{j,k,m}\widehat{\psi_{j,k,m}}(\varepsilon l) + \sum \limits_{(\widetilde{j},\widetilde{k}, \widetilde{m}) \in \widetilde{\Lambda}_{J}}  \gamma_{\widetilde{j},\widetilde{k},\widetilde{m}} \widehat{\psitilde_{\widetilde{j},\widetilde{k},\widetilde{m}}}(\varepsilon l) \bigg|^2 \nonumber \\
	&\leq  \sum \limits_{l \in (I_M)^c} \frac{1}{A_N} \Bigg( \sum \limits_{m' \in \Omega} \left| \widehat{\phi_{m'}}(\varepsilon l)\right|^2  + \sum \limits_{(j,k,m)\in \Lambda_{J} } \left| \varepsilon \widehat{\psi_{j,k,m}}(\varepsilon l) \right|^2+   \sum \limits_{(\widetilde{j},\widetilde{k}, \widetilde{m}) \in \widetilde{\Lambda}_{J}} \left| \varepsilon \widehat{\phi_{m'}}(\varepsilon l)\right|^2 \Bigg).
\end{align*}
By Proposition \ref{prop:EffFreqSupp} there exists $S(\theta, r, \varepsilon) $ so that we can choose
$$
M_i = \left\lceil S(\theta,r, \varepsilon) 2^{J(1+\delta)}A_N^{\frac{1}{2r-1}} \right\rceil \in \N, \quad i =1,2
$$
 to conclude
\begin{align*}
	\sum \limits_{l \in (I_M)^c}  |\langle f, s_l \rangle |^2 \leq \frac{\theta^2-1}{\theta^2}
\end{align*}
which  shows \eqref{claim:1} and hence
\begin{align*}
	\| P_{\Scal_M^{\eps}} f \| \geq \frac{1}{\theta}.
\end{align*}
\end{proof}

\subsection{Proof of Corollary \ref{cor:ErrorEstimates}}

For the proof of Corollary \ref{cor:ErrorEstimates} we recall the following results. By Theorem \ref{theorem:OptimalSparseApprox} the best $N$-term approximation $f^s_N$ for shearlets of cartoon-like functions obeys
\begin{align*}
	\| f - f^s_N \| \lesssim N^{-1} (\log N)^{3/2},
\end{align*}
and for wavelets the best $N$-term approximation $f_N^w$ obeys
\begin{align*}
	\| f - f^w_N \| \lesssim N^{-1/2} ,
\end{align*}
see \cite{ShearletBook}. Furthermore, these bounds are sharp, i.e. there exist cartoon-like functions such that the above inequalities are becoming equalities up to a constant. 

Now, let $N \in \N, f_N^s, f_{\sigma(N)}^w, \Rcal_{N_J}^s, \Rcal_{\sigma({N_J})}^w$ be as in the corollary. By Theorem \ref{maintheoremshearlets} there exist for fixed $\theta >1$ an $M=(M_1,M_2)$ and $M'=(M_1', M_2')$ in $\N \times \N$ such that
\begin{align*}
	c_{{N_J},M} = \inf_{\substack{ f \in \Rcal_{{N_J}}^s \\ \| f \| = 1}} \| P_{\Scal_M^{(\eps)}} f \| \geq \frac{1}{\theta}
\end{align*}
and
\begin{align*}
	c_{\sigma({N_J}),M'} = \inf_{\substack{ f \in \Rcal_{\sigma({N_J})}^w \\ \| f \| = 1}} \| P_{\Scal_{M'}^{(\eps)}} f \| \geq \frac{1}{\theta}.
\end{align*}
Moreover, due to Theorem \ref{maintheoremshearlets} and the linear stable sampling rate for wavelets proven in \cite{AHP1, AHKM} we can choose $M_1 \cdot M_2$ and $M'_1 \cdot M'_2$ to be of order $\sigma(N_J)$.

Using Theorem \ref{thm:Existence} we obtain
\begin{align*}
	\| f - G^w_{\sigma({N_J}), M'}(f) \| \lesssim \| f - P_{\Rcal_{\sigma({N_J})}^w}(f) \| \lesssim \| f - f_{\sigma(N)}^w \|Ê\lesssim \sigma(N)^{-1/2}.
\end{align*}
Analogously, for shearlets we obtain using Theorem \ref{thm:Existence}
\begin{align*}
	\| f - G^s_{{N_J}, M}(f) \|^2 \lesssim \| f - P_{\Rcal_{{N_J}}^s}(f) \|^2 \lesssim \left \| f - \sum_{\lambda \in I_N} \langle f, \psi^d_\lambda \rangle \psi_\lambda \right\|^2 \lesssim \sum_{\lambda \notin I_N} |\langle f, \psi^d_\lambda \rangle|^2,
\end{align*}
where $(\psi^d_{\lambda})_\lambda$ is the dual shearlet system of $(\psi_\lambda)_\lambda$ and $I_N$ denotes the index set of the $N$ largest coefficients $(\langle f, \psi_\lambda)_\lambda)$ with respect to the best $N$-term approximation $f_N^s$.

Furthermore, if the regularity is sufficiently large, it was shown in \cite{ParLoc} that we can relate the analysis coefficients of the primal frame with those of the dual frame by the relation
\begin{align*}
	(\langle f, \psi_ \lambda \rangle)_\lambda = G^\dag (\langle f, \psi^d_\lambda \rangle)_\lambda,
\end{align*}
where $G^\dag$ is the pseudo inverse of the Gramian operator associated to the shearlet system $(\psi_\lambda)$. Therefore the claimed rate follows from the decay rate of the shearlet coefficients. Hence,
\begin{align*}
	\| f - G^s_{{N_J}, M}(f) \|^2 \lesssim \sum_{\lambda \notin I_N} |\langle f, \psi_\lambda \rangle|^2 \lesssim N^{-2} (\log N )^3
\end{align*}
which yields the result.

\section{Numerics}\label{section:Numerics}

\subsection{Inequality \eqref{eq:asymptotics}}

In light of Corollary \ref{cor:ErrorEstimates} we want to compare the bounds $N^{-1}(\log N)^{3/2}$ and $\sigma(N)^{-1/2}$ with $\sigma$ defined as in \eqref{eq:sigmafunc} for $N \in \N$. We therefore aim to verify \eqref{eq:asymptotics} by showing the following behavior of the lower frame bound $A_N$ 
\begin{align}
	N^{-(1-\delta)/2}(\log N)^{3/2} \lesssim A_N^{1/(2r-1)}. \label{eq:lowerFrameBound}
\end{align}
For convenience we will check \eqref{eq:lowerFrameBound} in levels, i.e. the lower frame bounds $A_N$ will always belong to full shearlet systems at a scale $J$. Further, we use the \textsc{Shearlab} package which is available at
\begin{center}
\texttt{http://www.shearlab.org}
\end{center}
to generate shearlet systems with respect to different levels and their respective lower frame bounds. We wish to mention that both objects are available in the implementation.

Now, note that a shearlet system consists of approximately $2^{2J}$ many elements for a fixed scale $J$ where we ignore constants that are independent of $J$. We will numerically test \eqref{eq:asymptotics} by plotting the lower frame bounds of a full shearlet system for increasing $J$. This yields a monotonically decreasing curve as the scale increase. This will now be compared to the decay of the function $f(N)=N^{-(1-\delta)/2}(\log N)^{3/2}$ with $N = 2^{2J}$ for different values of $r$ and $\delta$ and for increasing $J$.

The informed reader might notice, that choosing $N = 2^{2J}$ is significantly  moderate, since the constant which determines the number of shearlets at scale $J$ is in practice much larger than 1. However, for larger $N$ the function $f$ decreases even faster. Moreover, since we are only interested in an asymptotic behavior of the lower frame bound with respect to increasing scales this is not a restriction. Indeed, already for that case it is strongly indicated that \eqref{eq:asymptotics} hold asymptotically as depicted in Figure \ref{fig:asymptotics}.

\begin{figure}[htp]
\begin{subfigure}{0.5\textwidth}
\hspace*{-1cm}
	\includegraphics[scale=.6]{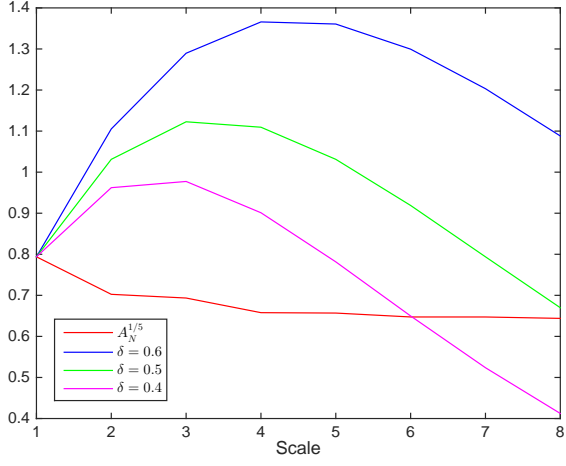}\hspace*{-1.2cm}\label{fig:Ratio1}
	\caption{$r = 3$ and $\delta = 0.4, 0.5, 0.6$}
\end{subfigure}
\begin{subfigure}{0.5\textwidth}
\hspace*{-.5cm}
	\includegraphics[scale=.6]{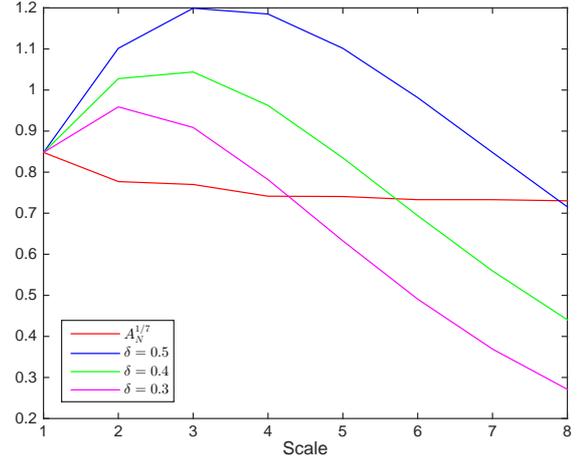}\hspace*{-0.8cm}\label{fig:Ratio2}
	\caption{$r = 4$ and $\delta = 0.3, 0.4, 0.5$}
\end{subfigure}

\begin{subfigure}{0.5\textwidth}
\hspace*{-1cm}
	\includegraphics[scale=.6]{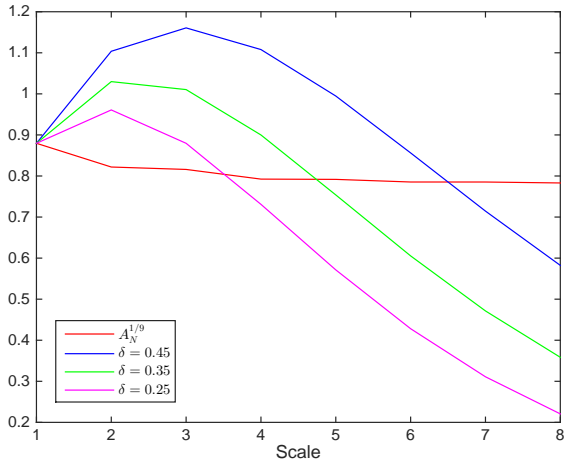}\hspace*{-1.2cm}\label{fig:Ratio3}
	\caption{$r = 5$ and $\delta = 0.25, 0.35, 0.45$}
\end{subfigure}
\begin{subfigure}{0.5\textwidth}
\hspace*{-.5cm}
	\includegraphics[scale=.6]{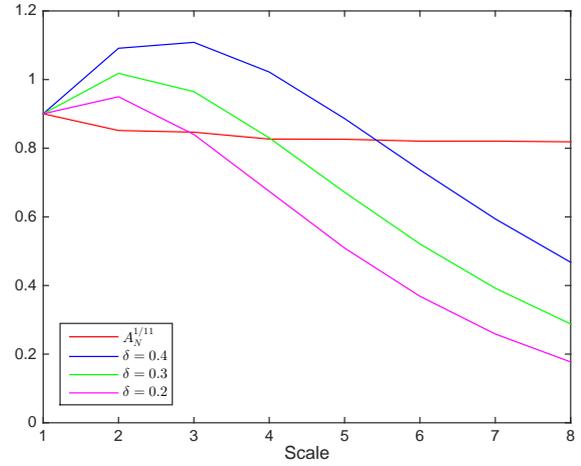}\hspace*{-0.8cm}\label{fig:Ratio4}
	\caption{$r = 6$ and $\delta = 0.2, 0.3, 0.4$}
\end{subfigure}
	\caption{Plot of the behavior of lower frame bound and the critical function $N^{-(1-\delta)/2}(\log N)^{3/2}$ for different values of $r$ and $\delta$}\label{fig:asymptotics}
\end{figure}

\subsection{Reconstruction from MR data using shearlets}

We further provide some reconstructions from MR data using complex exponentials (Fourier inversion), compactly supported Daubechies wavelets and compactly supported shearlets. Although our main result guarantees stable and convergent reconstructions when the sampling rate is almost linear it is not efficient (and also not necessary) to acquire that many samples in practice. In our numerics we will subsample the Fourier data to achieve practical relevance. All computation are done in \textsc{Matlab}.

\subsection{Data}

The underlying Fourier data or also called \emph{k-space} was acquired using a multi-channel acquisition consisting of four channels. Each of the four k-spaces has a $128 \times 128$ image resolution. Moreover, each channel data results in a single image of same size, e.g. by applying an inverse Fourier transform, see Figure \ref{Subfig:4chFourier} and Figure \ref{Subfig:4ch}. 

The  single channel images are combined by using the sum-of-squares method \cite{SENSE}.  The resulting sum-of-squares image from all four channels of the original k-space is used as the reference image, cf. Figure \ref{Subfig:Orig1Brain128}.

\begin{figure}[H]
        \begin{subfigure}[b]{0.5\textwidth}
                	\centering
%                	\hspace*{-0.5cm}
		\includegraphics[scale=0.36]{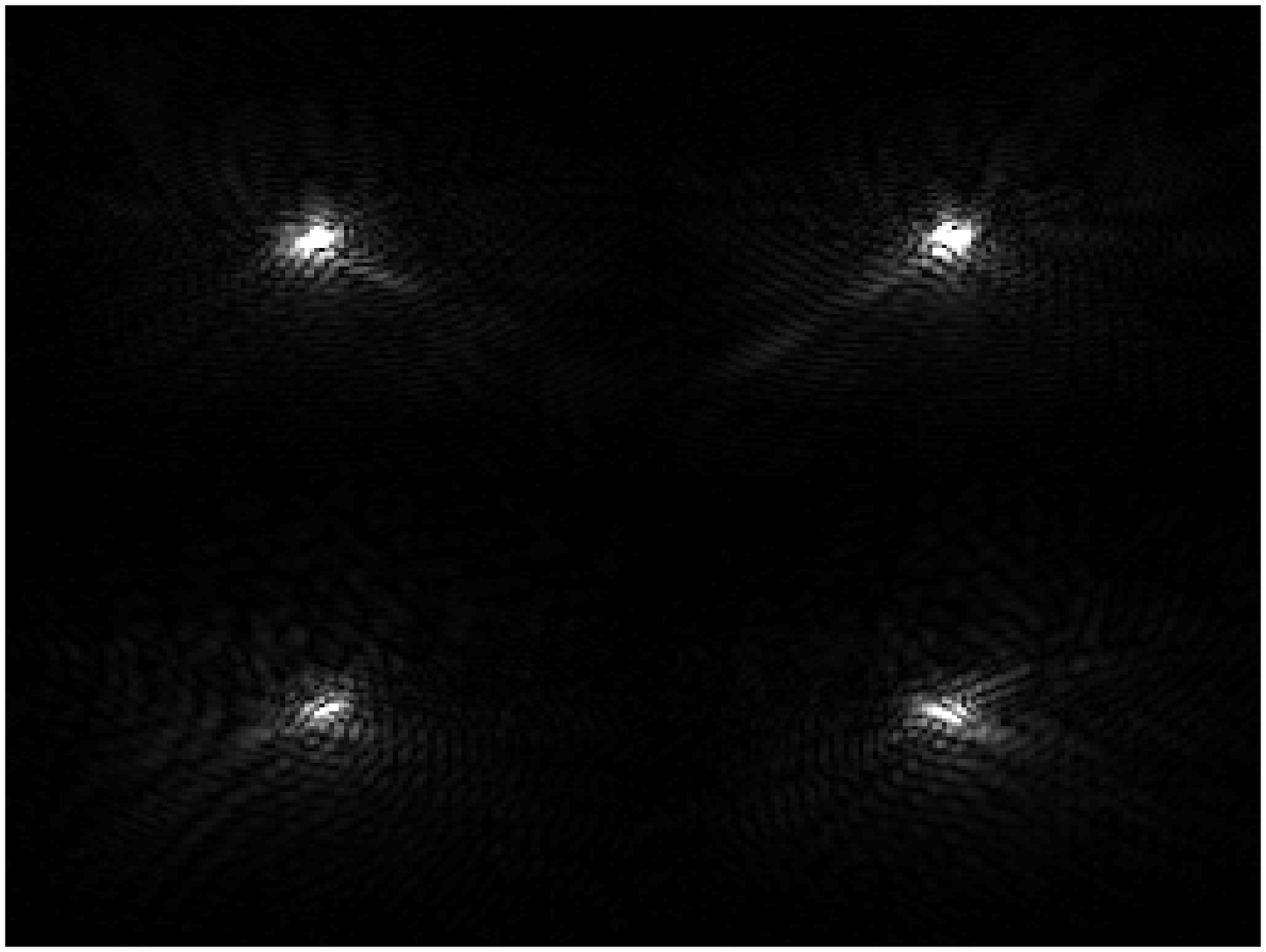}
		\caption{4 channels input data}\label{Subfig:4chFourier}
        \end{subfigure}        
        \begin{subfigure}[b]{0.5\textwidth}
                	\centering
%                	\hspace*{-0.5cm}
		\includegraphics[scale=0.36]{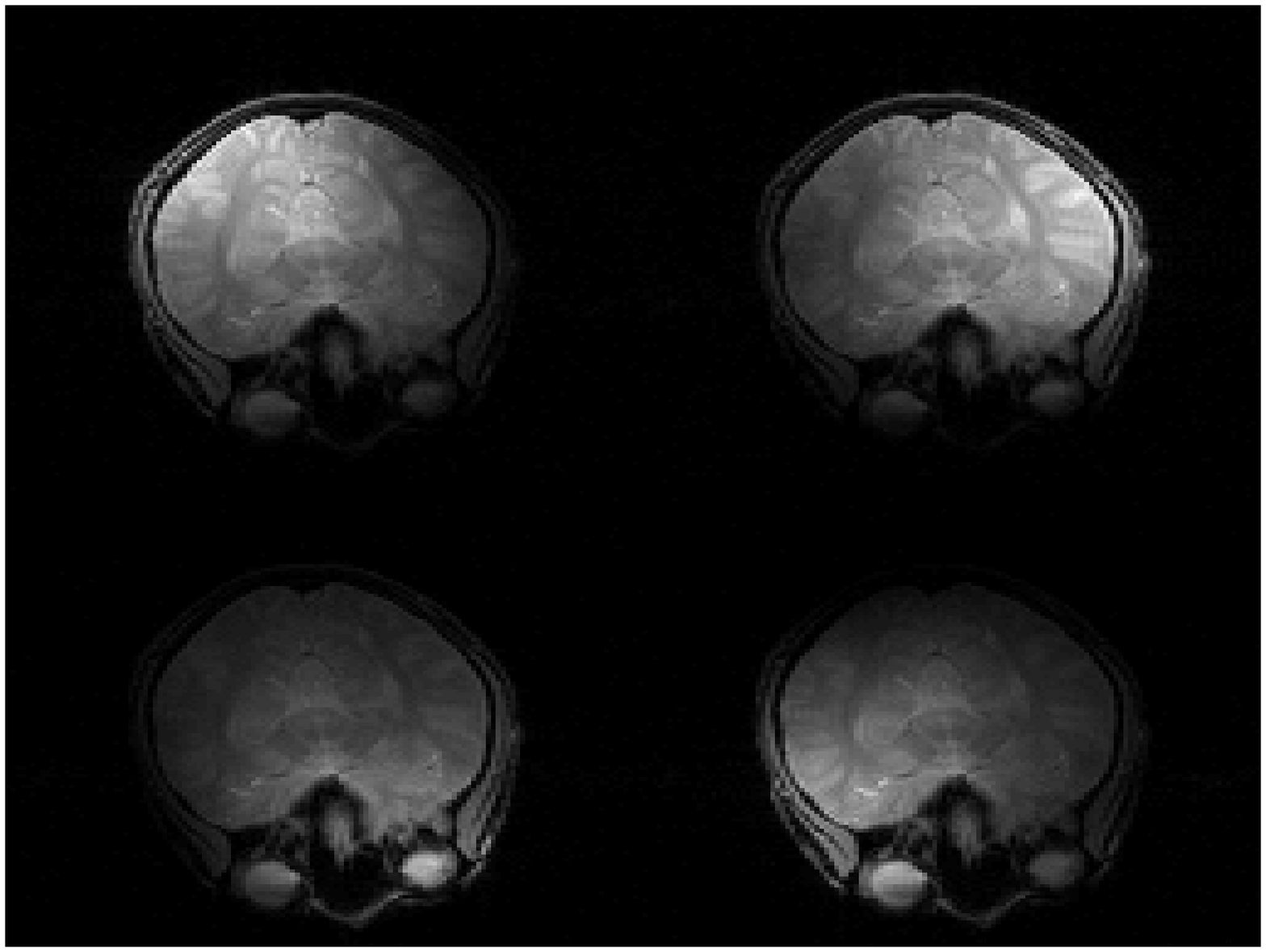}
		\caption{4 channel reconstructions}\label{Subfig:4ch}
        \end{subfigure}        
\begin{center}
        \begin{subfigure}[b]{0.5\textwidth}
                	\centering
		\includegraphics[scale=0.45]{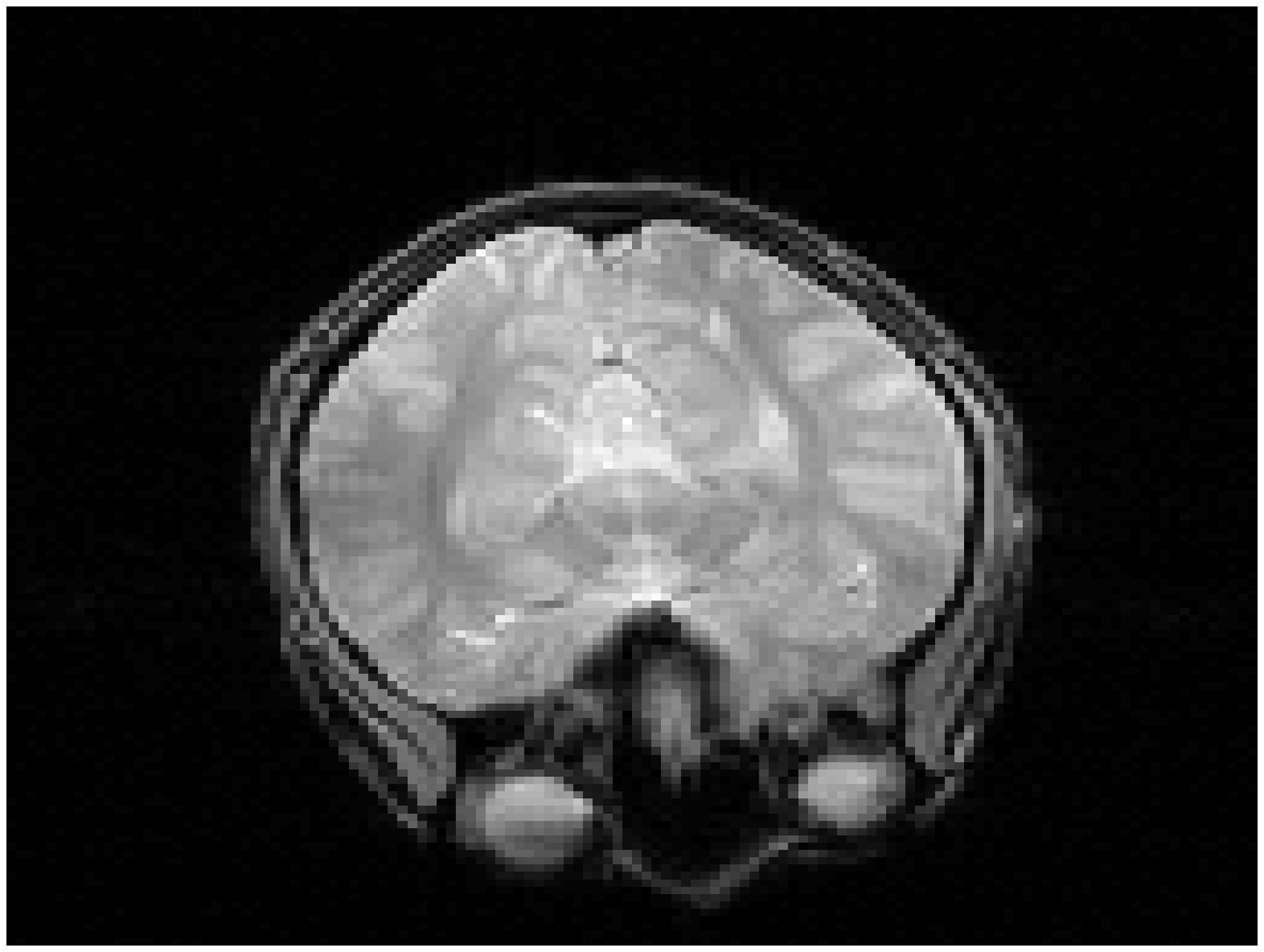}
		\caption{Combined image by sum-of-squares}\label{Subfig:Orig1Brain128}
\end{subfigure}
\end{center}
        \caption{Reference images computed with complete data}\label{fig:reference}
\end{figure}

\subsection{Methods}

We mimic the acceleration of factor 5 by using a $0-1$ mask to subsample the k-space, cf. Figure \ref{Subfig:Spiral} and Figure \ref{Subfig:Radial}. We have chosen a phyllotaxis spiral which is also used in \cite{WanZenErtMueNad} and a standard radial subsampling scheme.

The final reconstruction is obtained by computing each channel image separately and then combine it by using the \emph{sum-of-squares method} \cite{SENSE}. More precisely, if $I_k$ denotes the $k$-th image reconstructed from the $k$-th k-space then the final image $I$ is computed by
\begin{align*}
	I = \sqrt{ |I_1| + \ldots + |I_N|^2},
\end{align*}
where $N$ is the total number of channels and the the square operation, as well as the modulus is to be understood entry wise for the matrices $I_k$.

The reconstruction of the single channel images are computed as follows.
\begin{itemize}
\item Shearlets: 

The reconstructions are obtained by solving the following unconstrained minimization problem
\begin{align*}
  \min_{u} \frac{1}{2} \|\Fcalt u  - b \|_2^2 + \lambda \| \Scal \Hil(u) \|_1,  
\end{align*}
where $\Fcalt$ is the subsampled Fourier operator, $b$ is the subsampled input data, $\Scal \Hil$ is the shearlet transform and $\lambda$ are some parameters larger than zero, and $u$ is the image. 

The penalty function $\| \Scal \Hil(u) \|_1$ enfores sparsity with respect to the shearlet system. The code for the shearlet transform is downloadable at
\begin{center}
\texttt{http://www.shearlab.org}
\end{center}
\item Wavelets: 

Reconstructions are computed using the  SparseMRI package from Lustig et al. \cite{LusDonPau} which is downloaded from
\begin{center}
\texttt{http://www.eecs.berkeley.edu/\textasciitilde{}mlustig/Software.html}
\end{center}
This method is a standard tool in MRI to compute wavelet sparsity-based reconstructions from Fourier measurements.
\item Fourier: 

The reconstructions using complex exponentials are computed by taking the Fourier inversion of the subsampled data.
\end{itemize}

\subsection{Results}

The reconstructions show less artifacts and improved image quality using shearlets compared to the reconstruction obtained from compactly supported wavelets, cf. Figure \ref{Fig:Brain128}. Although the theory of optimal sparse approximation rates hold in a continuous setup, it is surprising to see such strong differences for images of such low resolutions such as $128 \times 128$.

\begin{table}[h]
\centering
\begin{tabular}{@{}|c|c|c|c|}
\hline
&Shearlets& Wavelets & Fourier inversion\\
\hline 
Spiral mask&0.638& 0.1006& 0.2570\\
\hline
Radial mask&0.0847& 0.1282& 0.2690\\
\hline
\end{tabular}
\caption{Relative error for Figure \ref{Fig:Brain128}}\label{table:Error}
\end{table}

\begin{figure}[H]
        \begin{subfigure}[b]{0.5\textwidth}
                	\centering
%                	\hspace*{-0.5cm}
		\includegraphics[scale=0.45]{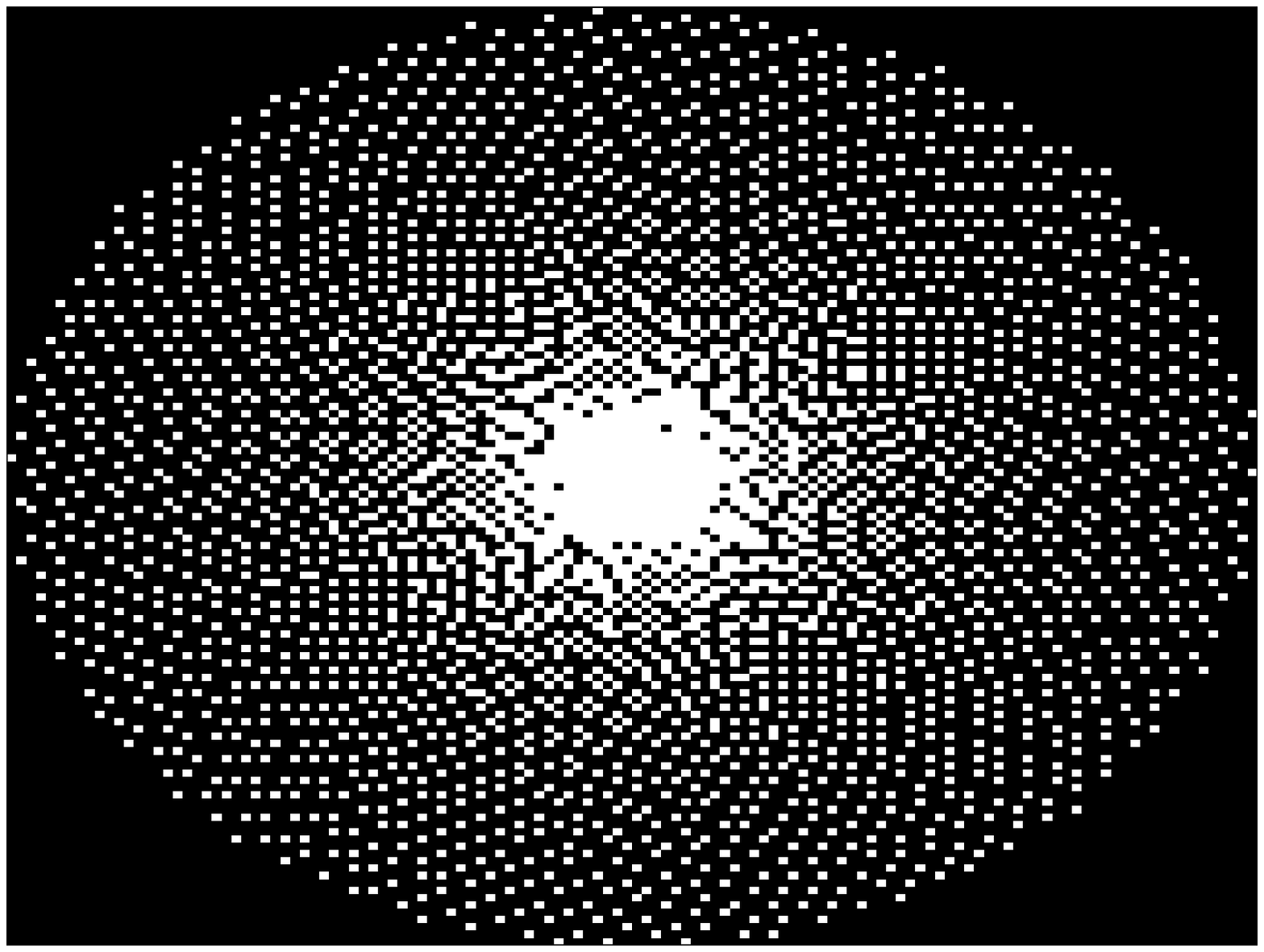}
		\caption{Phyllotaxis spiral mask: 20.37 \%}\label{Subfig:Spiral}
        \end{subfigure}        \quad
        \begin{subfigure}[b]{0.5\textwidth}
                	\centering
%                	\hspace*{-0.5cm}
		\includegraphics[scale=0.45]{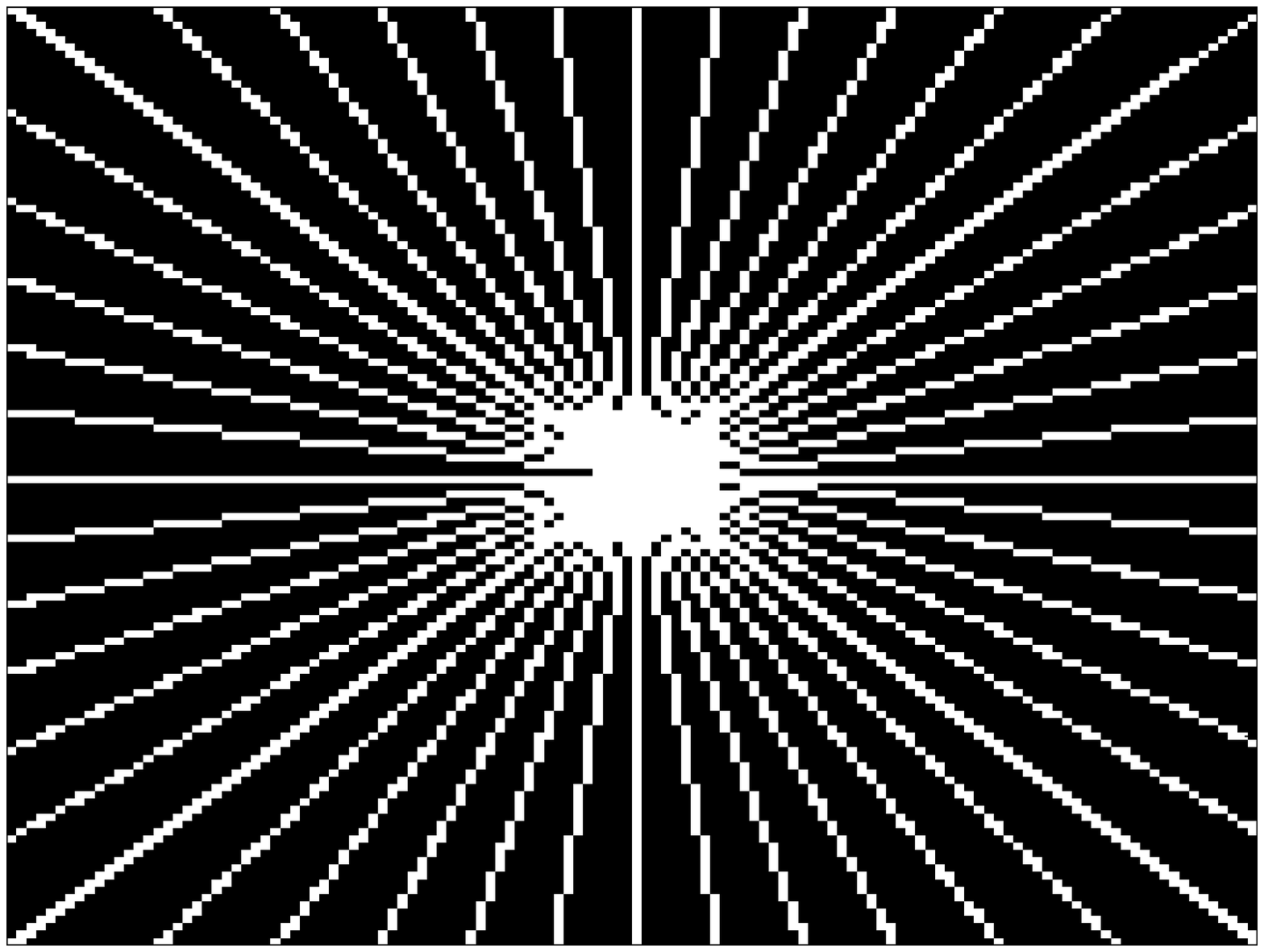}
		\caption{Radial mask: 20.74\%}\label{Subfig:Radial}
        \end{subfigure}
        
	\begin{subfigure}[b]{0.5\textwidth}
		\centering
%		\hspace*{-0.5cm}
		\includegraphics[scale=0.45]{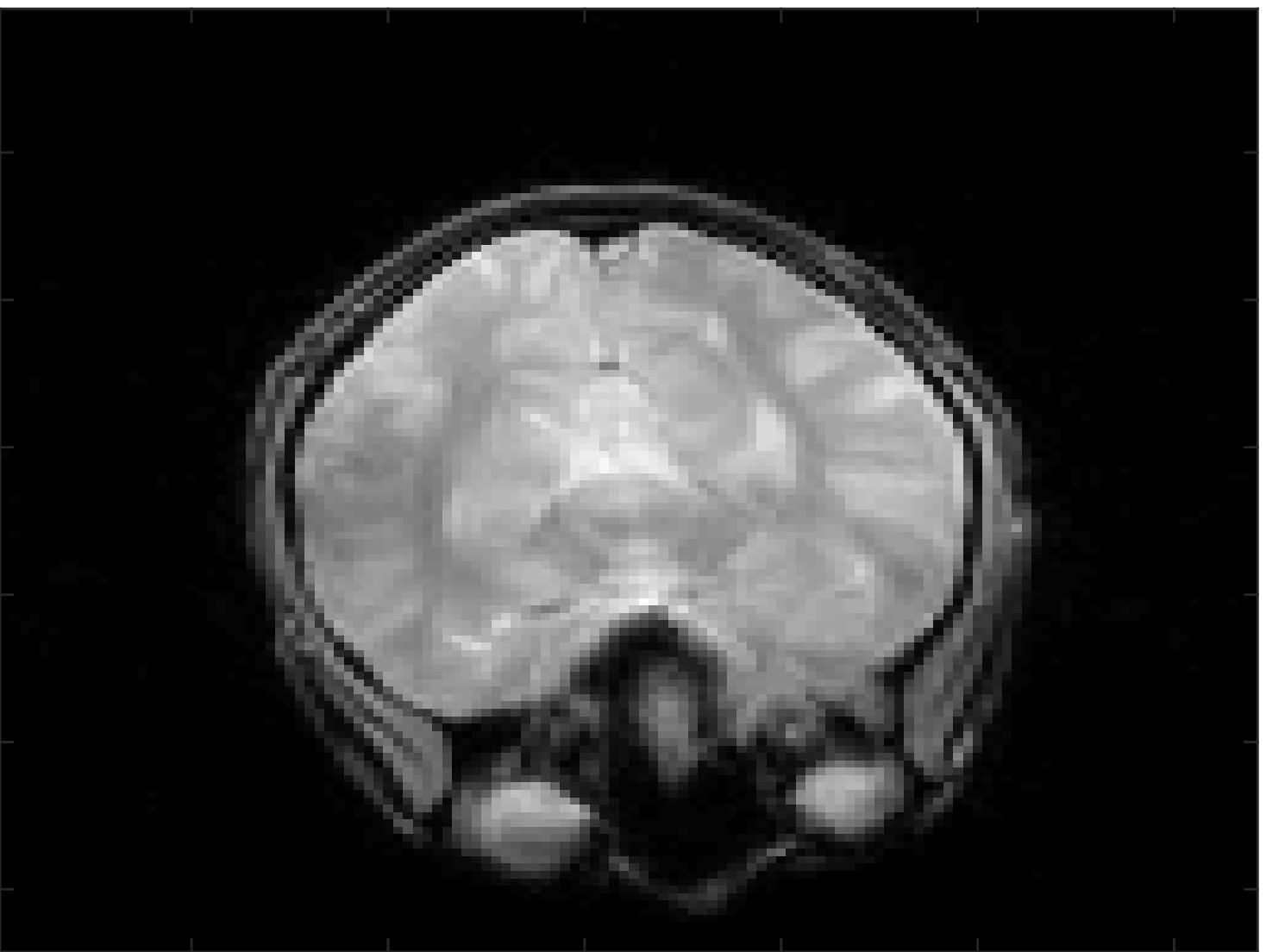}
		\caption{Reconstruction: shearlets, spiral}\label{Subfig:shearSp}
        \end{subfigure} \quad
        \begin{subfigure}[b]{0.5\textwidth}
		\centering
%		\hspace*{-0.5cm}
		\includegraphics[scale=0.45]{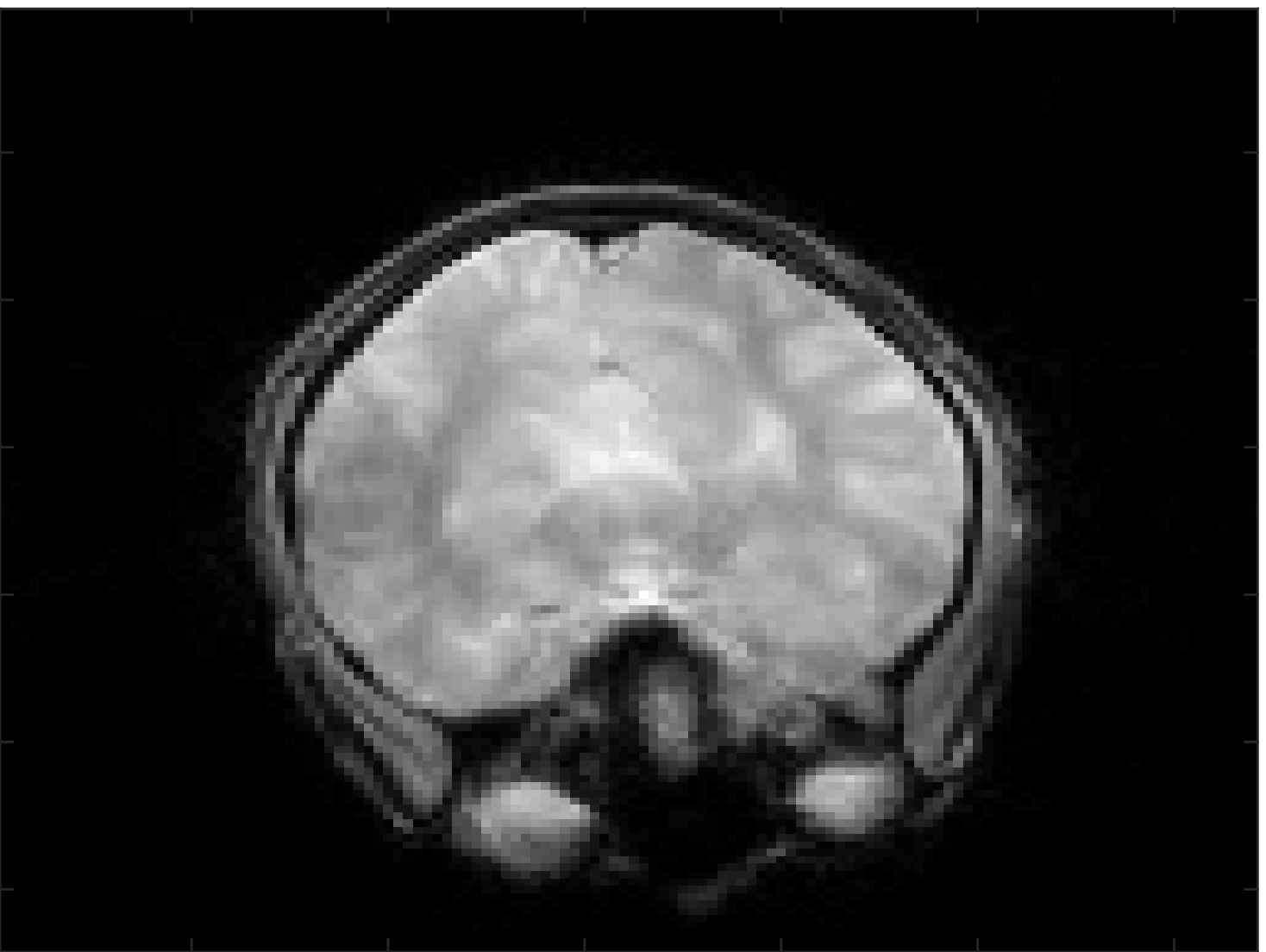}
		\caption{Reconstruction: shearlets, radial }\label{Subfig:shearRa}
        \end{subfigure}
        
	\begin{subfigure}[b]{0.5\textwidth}
		\centering
%		\hspace*{-0.5cm}
		\includegraphics[scale=0.45]{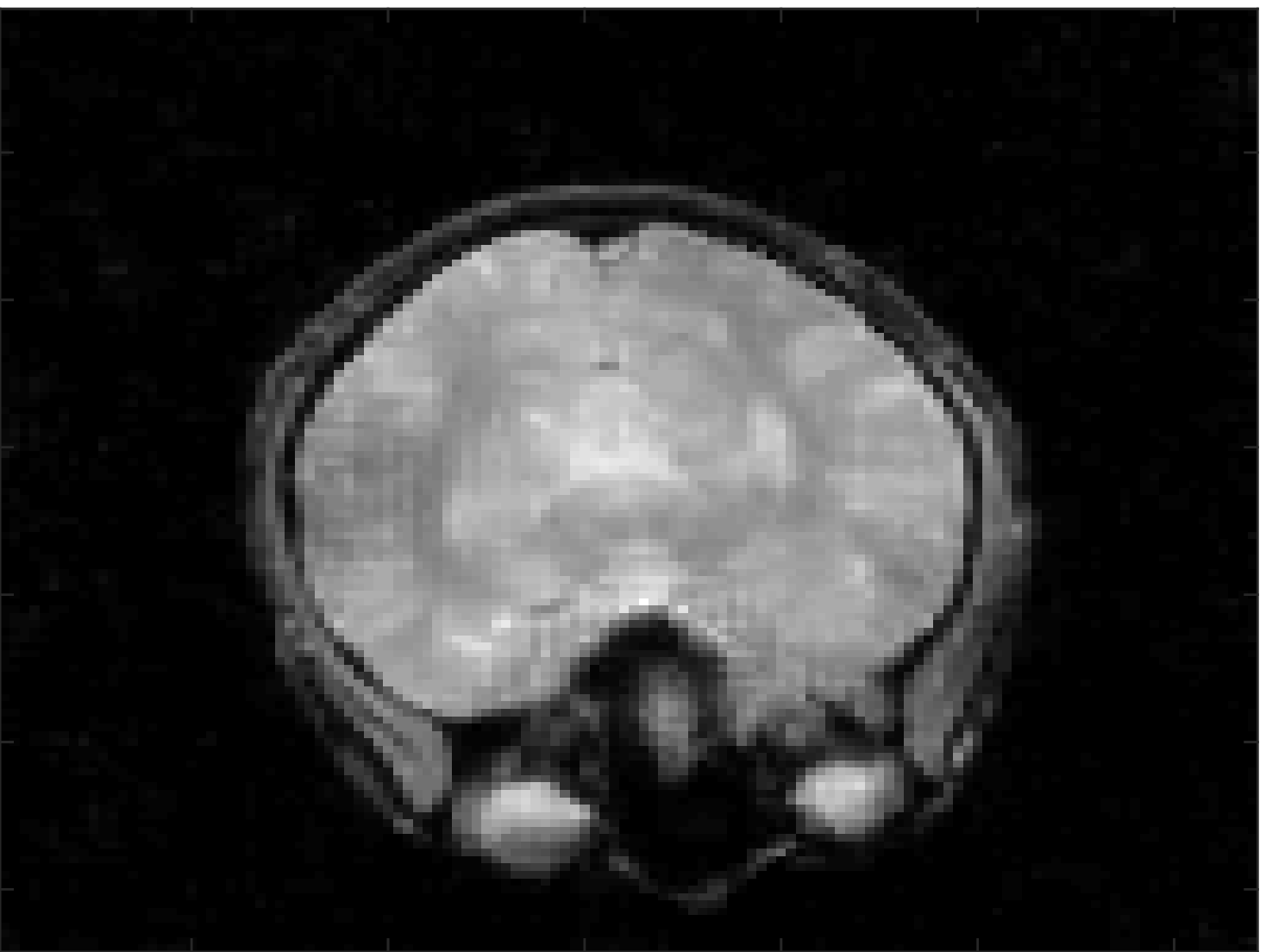}
		\caption{Reconstruction: wavelets, spiral}\label{Subfig:waveSp}
        \end{subfigure} \quad
        \begin{subfigure}[b]{0.5\textwidth}
		\centering
%		\hspace*{-0.5cm}
		\includegraphics[scale=0.45]{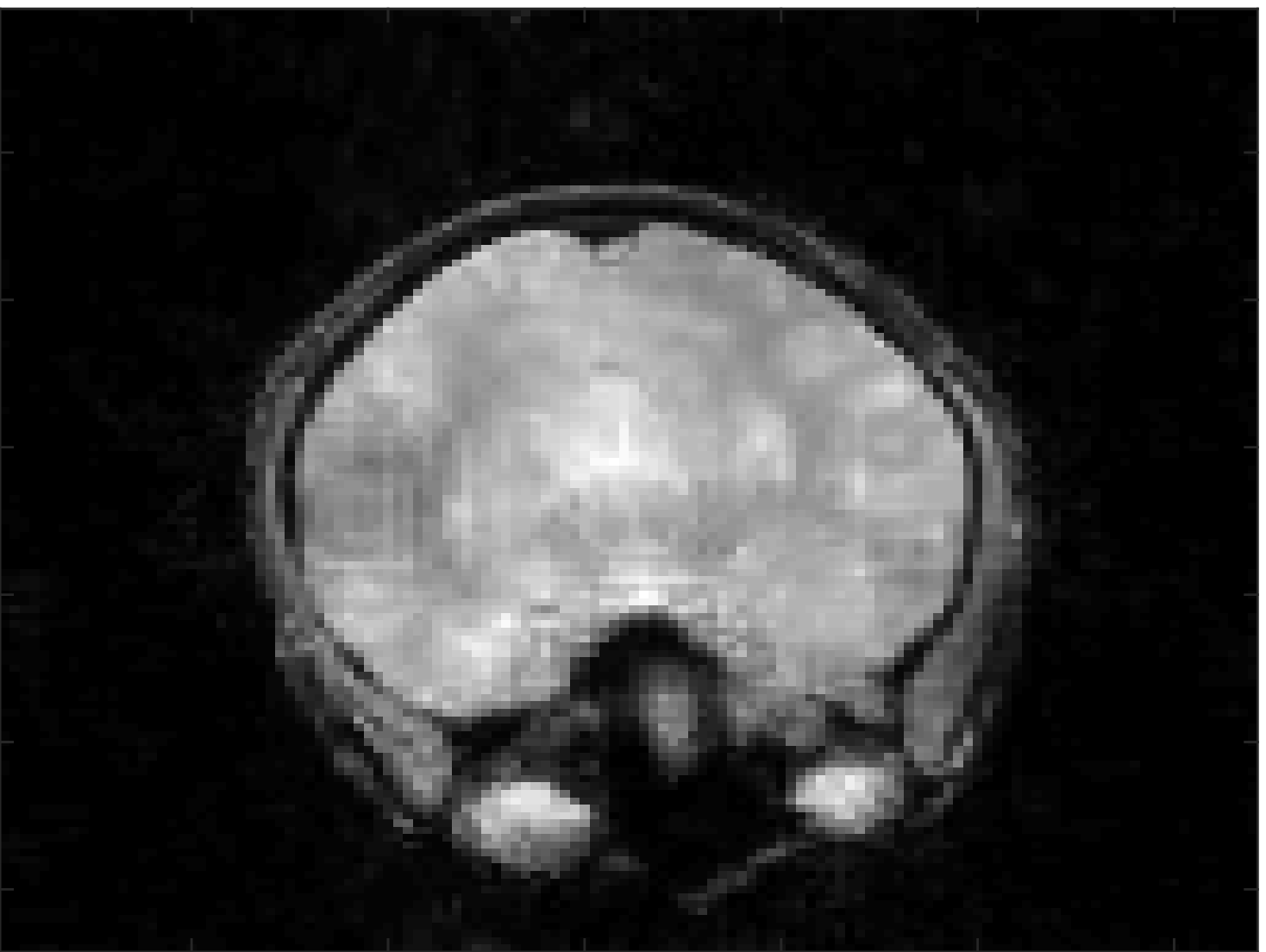}
		\caption{Reconstruction: wavelets, radial}\label{Subfig:waveRa}
        \end{subfigure}
                
	\begin{subfigure}[b]{0.5\textwidth}
		\centering
%		\hspace*{-0.5cm}
		\includegraphics[scale=0.45]{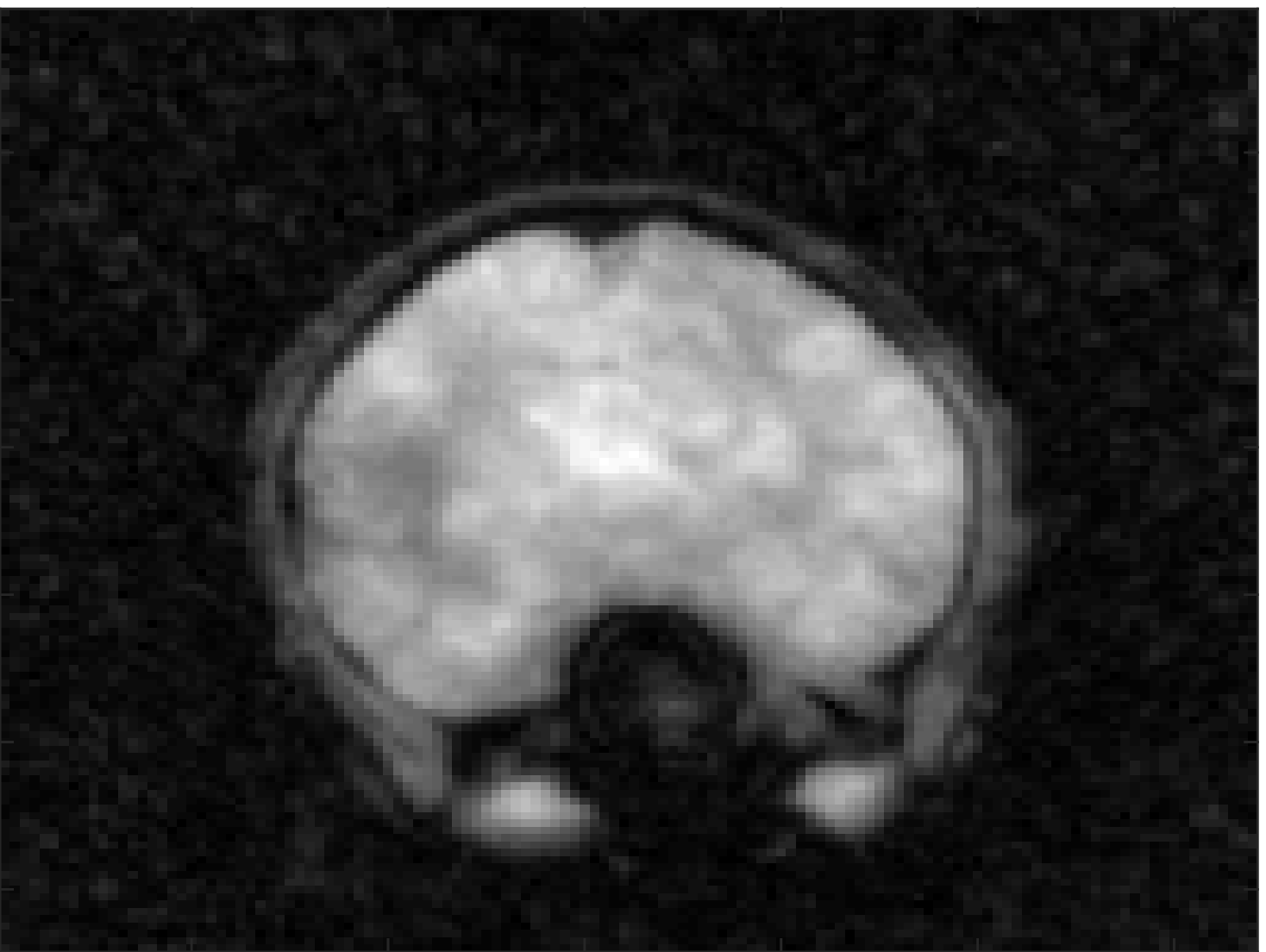}
		\caption{Reconstruction: Fourier inversion, spiral}\label{Subfig:fourierSp}
        \end{subfigure} \quad
        \begin{subfigure}[b]{0.5\textwidth}
		\centering
%		\hspace*{-0.5cm}
		\includegraphics[scale=0.45]{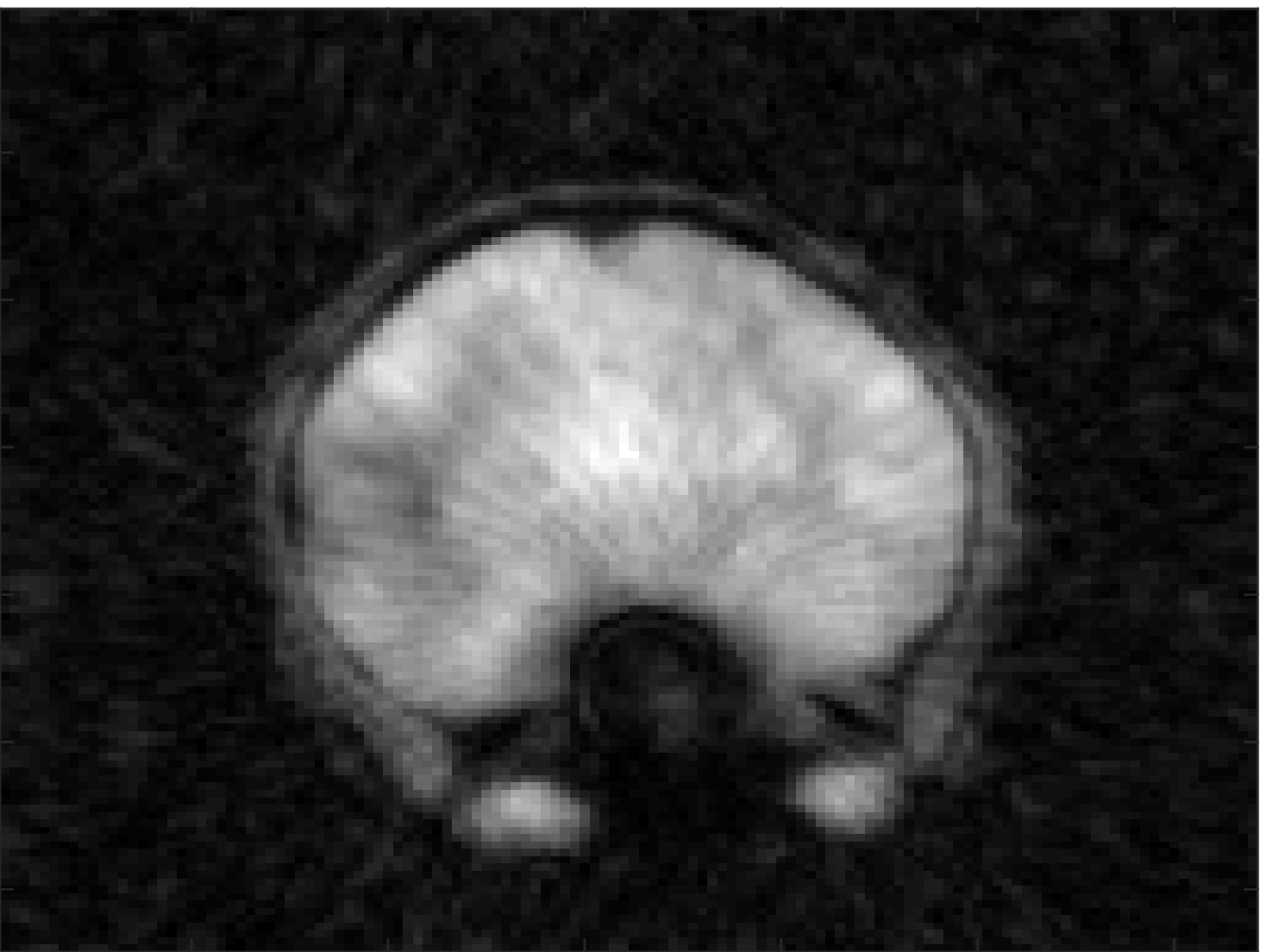}
		\caption{Reconstruction: Fourier inversion, radial}\label{Subfig:fourierRa}
        \end{subfigure}        
        \caption{Reconstruction using a spiral and radial mask.} \label{Fig:Brain128}
\end{figure}

\section{Future work}\label{sec:futurework}

We like to address multiple extensions and future directions of the work presented in this paper.

\subsection{Linear stable sampling rate}

In this paper we have established a stable sampling rate for sampling with respect to complex exponentials and recovery using compactly supported shearlet frames. Our rate essentially differs from the rate for wavelets by an additional $N^\delta$ term and the lower frame bound. Although it seems plausible that the rate should depend on the stability of the redundant system, measured by the lower frame bound, it is interesting to see, whetherthe could be improved to actual linearity up to a dependency of the lower frame bound, i.e. erasing $N^\delta$ while keeping $A_N$. Furthermore, efficient numerical algorithms to compute the stable sampling rate for shearlets are currently missing. This issue is left for future work.

\subsection{Nonuniform sampling}

Moreover, we have restricted ourselves to the case of a uniform sampling pattern, in fact, the samples were assumed to be acquired on a grid. Although this is possible in practice, it is very unusual. Many MRI acquisition schemes are nowadays non-uniform schemes, e.g. spiral patterns or radial patterns are often used, see Figure \ref{Subfig:Spiral} and Figure \ref{Subfig:Radial}. In \cite{AGH1, AGH2} the authors introduced \emph{non-uniform generalized sampling} and proved stable reconstructions are possible when using \emph{folded, periodic,} and \emph{boundary} wavelets. We strongly believe that the methods used in the present paper can be used to obtain similar results for shearlets in the non-uniform case as well. This is will be studied in an upcoming work.

\subsection{Parabolic molecules}

Shearlets are a special type of \emph{parabolic molecules} \cite{ParMol}. It is therefore an interesting question whether our results can be extended to such systems. We explain briefly how such generalizations can be tackled.

First of all, the reconstruction space $\Rcal_N$ is spanned by the first $N$ elements of the representation system. In our case of compactly supported shearlets it was build by considering all atoms whose support intersect a region of interest. If one considers atoms that are not compactly supported, then one has to make a careful choice of ordering the representation system and select the first $N$ elements of the system. However, the argument for when a reconstruction from Fourier measurements is possible, is solely based on very good frequency localization of the atoms. Thus this argument is applicable to other systems such as wavelet frames or general parabolic molecules. However, for a qualitative judgement of the results one needs to relate the lower frame bound of the finite system to the approximation rate, similar as we did in \eqref{eq:asymptotics}.

\subsection{Compressed sensing}

Another perspective that is indispensable for reconstruction problems from Fourier measurements such as in MRI is \emph{compressed sensing} cf. \cite{CanRomTao, Don}. Moreover, the new concepts of \emph{asymptotic sparsity} and \emph{asymptotic incoherence} introduced in \cite{AdcHanPooBog} are perfectly suited for shearlets, in fact natural images are asymptotically sparse in shearlets and it can be shown that the cross gramian between complex exponentials and shearlets is asymptotically incoherent. Further investigations in this direction are also possivle for future work.

As shown in the numerics, shearlets have a great potential in recovering real medical images from their subsampled Fourier data. The development of a sophisticated $\ell^1$ algorithm is therefore of high importance. A mathematical analysis and a comparison of different recovery algorithms will be provided in an upcoming paper.

\section*{Acknowledgements}
The author would like to thank Gitta Kutyniok for helpful discussions and the anonymous reviewers for critical comments that greatly improved the quality of the paper. Moreover, the author thanks Dr. Sina Straub from German Cancer Research Center (DKFZ) for providing the MRI data and Dr. Matthias Dieringer from Siemens AG, Healthcare Sector for providing a code to read the raw data into MATLAB. Further, the author acknowledges support from the Berlin Mathematical School as well as the DFG Collaborative Research Center TRR 109 "Discretization in Geometry and Dynamics". 

%%%%%%%%%%%%%%%%%%%
%%%%%% REFERENCES %%%%%
%%%%%%%%%%%%%%%%%%%
\bibliographystyle{abbrv}
\bibliography{ShearletGSRef-1}

\end{document}